\documentclass[a4paper,10pt,sort&compress]{elsarticle}
\usepackage{amscd}
\usepackage{bm}				
\usepackage{mathtools}		
\usepackage{booktabs}		
\usepackage{float}			
\usepackage{amssymb}
\usepackage{mathrsfs}
\usepackage{dsfont}
\usepackage{appendix}
\usepackage{mathrsfs}
\usepackage{graphicx}
\usepackage{epstopdf}
\usepackage{subfigure}
\usepackage{color}
\usepackage{latexsym}
\usepackage{amsmath}
\usepackage{amsfonts}
\usepackage{amssymb}
\usepackage{epsf}
\usepackage{graphicx}
\usepackage{epsf}
\usepackage{stmaryrd}
\usepackage{multicol}
\usepackage{subfigure}
\usepackage{ifpdf}
\usepackage{amsmath}
\usepackage{amsfonts}
\usepackage{amsthm}
\usepackage{mathrsfs}
\usepackage{color}
\usepackage{graphicx}
\usepackage{CJK}
\usepackage{graphicx,natbib,amssymb,lineno}
\graphicspath{{figure/}}                   %% Path of figure (Modified by B.Cao)
\ifpdf
\usepackage[%
  pdftitle={Instructions for use of the document class
    elsart},%
  pdfauthor={},%
  pdfsubject={The preprint document class elsart},%
  pdfkeywords={instructions for use, elsart, document class},%
  pdfstartview=FitH,%
  bookmarks=true,%
  bookmarksopen=true,%
  breaklinks=true,%
  colorlinks=true,%
  linkcolor=blue,anchorcolor=blue,%
  citecolor=blue,filecolor=blue,%
  menucolor=blue,pagecolor=blue,%
  urlcolor=blue]{hyperref}
\else
\usepackage[%
  breaklinks=true,%
  colorlinks=true,%
  linkcolor=blue,anchorcolor=blue,%
  citecolor=blue,filecolor=blue,%
  menucolor=blue,pagecolor=blue,%
  urlcolor=blue]{hyperref}
\fi

\makeatletter
\def\elsartstyle{%
    \def\normalsize{\@setfontsize\normalsize\@xiipt{14.5}}
    \def\small{\@setfontsize\small\@xipt{13.6}}
    \let\footnotesize=\small
    \def\large{\@setfontsize\large\@xivpt{18}}
    \def\Large{\@setfontsize\Large\@xviipt{22}}
    \skip\@mpfootins = 18\p@ \@plus 2\p@
    \normalsize
} %\@ifundefined{square}{}{\let\Box\square} \makeatother

\newtheorem{theorem}{Theorem}[section]
\newtheorem{lemma}{Lemma}[section]

\theoremstyle{definition}

\numberwithin{equation}{section}
\newtheorem{remark}[theorem]{Remark}
\newcommand{\RNum}[1]{\uppercase\expandafter{\romannumeral #1\relax}}

\usepackage[text={6.5in,9in},centering]{geometry}

\journal{\textbf{~}}
\usepackage{amsthm}
\begin{document}

%MS++++++++++++++++++++++++++++++ Main Body Begin Here +++++++++++++++++++++++++++++
\title{ {\bf Finite horizon optimal control of reaction-diffusion SIV epidemic system with stochastic environment}}
%\author{Yanyan Du$^a$}{\ead{mduyanyan@163.com}}
%\footnote{$^*$ Author to whom any correspondence should be addressed.}}{\ead{wangruobing25@163.com}}
\author{Zong Wang$^*$
\footnote{$^*$ Author to whom any correspondence should be addressed.}}{\ead{ wangzong@qut.edu.cn}}
%\author{Ruobing Wang$^b$}
\address{ School of Science, Qingdao University of Technology, Qingdao, 266520, P.R. China}
%\address{$^b$  Department of Oncology, Affiliated Qingdao Central Hospital of Qingdao University, Qingdao Cancer Hospital, Qingdao, Shandong, China}
%\address{$^b$  Department of Mathematics, University of Florida, Gainesville, FL 32611, USA}
\linespread{1.5}

\begin{frontmatter}
\begin{abstract}
This contribution mainly focuses on the finite horizon optimal control problems of a susceptible-infected-vaccinated(SIV) epidemic system governed by reaction-diffusion equations and Markov switching. Stochastic dynamic programming is employed to find the optimal vaccination effort and economic return for a stochastic reaction diffusion SIV epidemic model. To achieve this, a key step is to show the existence and uniqueness of invariant measure for the model. Then, we obtained the necessary and sufficient conditions for the near-optimal control. Furthermore, we give an algorithm to approximate the Hamilton-Jacobi-Bellman (HJB) equation. Finally, some numerical simulations are presented to confirm our analytic results.% and to show the spatial heterogeneities are important factors for the global dynamics. %Numerical simulations are presented to display the theoretical results and verify the effect of treatment control on epidemic diseases. The near-optimal control is one problem in which the infected is higher at low cost.The near-optimal control is one problem in which the infected is higher at low cost.
\end{abstract}
\end{frontmatter}

\section{Introduction}

Recent research has revealed that epidemic models incorporating reaction-diffusion mechanisms are increasingly being used to depict the dynamics of epidemics in spatially heterogeneous contexts \cite{kiyt,8,5}. In real-world scenarios, the spread of infectious diseases is profoundly influenced by spatial heterogeneity. In the field of epidemiology, there is a growing body of evidence indicating that spatial diffusion and individual mobility play crucial roles in the dissemination of infectious diseases \cite{81,82}. For example, Feng et al. \cite{81} explored the stability of a stochastic epidemic model that includes diffusion. Similarly, M. Seplveda and M. Bendahmane \cite{82} examined the convergence of a finite volume scheme for nonlocal reaction-diffusion epidemic systems. To our knowledge, studies focusing on optimal control problems in SIV epidemic models governed by partial differential equations are scarce. Hence, investigating the dynamics of epidemic models that incorporate reaction-diffusion is of exceptional importance.

Optimal control strategies are crucial for the study of different diseases issues, which can lead to the effective control of specific infectious disease \cite{3x1}. Mathematically, the optimal control is obtained by solving the state equation and the adjoint equation or the HJB equation \cite{5z1}.  It is well known that it is difficult to find the analytical solution to the HJB equation due to the partial differential equation. Therefore, it is necessary to find the successive approximation method for the finite horizon optimal control problem of the stochastic epidemic system.

To effectively curb the development of the infectious disease, vaccination and isolation of confirmed individuals are the most common control measures. As we all known, vaccination is one of the most effective prevention measure against many diseases such as measles, polio, pertussis and influenza \cite{8..,8.i,800}. For example, Wei et al. \cite{8..} studied the vaccine induces antibodies to H1N1 and prevents the spread of the disease. Utazi and Tatem \cite{8.i} give the precise mapping reveals gaps in global measles vaccination coverage. However, little research has been done on the control, vaccination and treatment of infectious diseases.

In this paper, we give the approximation method for the finite horizon optimal control problem of the stochastic epidemic system and we also give the necessary and sufficient conditions for the near optimal solution of SIV model were given from the perspective of the minimum cost of disease control and the minimum number of infected persons. The main contributions and challenges of this paper are as follows:

\begin{itemize}
\item It is more difficult to investigate the reaction-diffusion, Markov swithcing, stochastic SIV epidemic model because the corresponding the epidemic system involve Markov term $\Lambda_{t}$ in \eqref{1`q1}. To over this difficulty, the method of steps and the generalized Gronwall inequality are used simultaneously to derive the invarient measure criterion of the system.
%\item Provision of an innovative reaction-diffusion stochastic SIV epidemic model with  Markov switching, vaccination and saturated treatment;
\item Finite horizon optimal control and the algorithm to approximate the Hamilton-Jacobi-Bellman (HJB) equation for our novel  stochastic SIV model with  Markov switching, vaccination and saturated treatment are given;
%\item Compared to the existing results about stochastic SIV epidemic system, \cite{902,903}, the ones of our paper are more complete, more correct and more general.
\end{itemize}
%The spatial diffusion, environment noise, Markov chains and vaccination saturated cure are taken into a SIV epidemic system consideration in this paper. The objective of the control problem is to decrease the number of the infectious with minimum cost of disease control. We establish the sufficient as well as necessary conditions of near optimal control problems.

%In fact, in the real world, population is spread out in space and is constantly spreading. In addition, we know that the existence and uniqueness of invariant measures is one of the important properties of random population models with Markov switching and spatial diffusion. However, if we introduce diffusion into the stochastic infectious disease model, the corresponding Kolmogorov-Fokker-Planck (KFP) equation becomes more complex.

The remaining part of the paper is so organized as: \autoref{Section2} give the preliminaries and introduce of stochastic partial differential equation SIV epidemic model; \autoref{Section3} prove the existence and uniqueness of invariant measures; \autoref{Section4}, institution of sufficient conditions and necessary conditions for near-optimal control; \autoref{Section5} analyze the optimal control and numerical algorithm; Some numerical experiments are carried out in \autoref{Section6}; Finally, we make a brief conclusion in \autoref{Section7}.

\section{Model and Preliminaries}\label{Section2}
\subsection{Model}

A stochastic susceptible-infected-vaccinated(SIV) model was put forward by Zhang et al. \cite{zhang1} as follows:
\begin{equation}\label{u}
\begin{cases}
\begin{split}
{dS(t)}&=\left((1-p)b+\alpha  I(t)-(\mu+u_1(t))S(t)-\beta S(t) I(t)\right)dt-\sigma S(t) I(t)dB_{1}(t),\\
{dI(t)}&= \left(\beta S(t) I(t)  +(1-e)\beta V(t) I(t)-(\mu+\alpha )I(t)-\frac{mu_{2}(t)I(t)}{1+\eta I(t)}\right)dt\\
&+\sigma S(t) I(t) dB_{2}(t)+(1-e)\sigma V(t) I(t) dB_{4}(t),\\
{dV(t)}&= \left( p b+u_1(t) S(t)-\mu V(t)-(1-e)\beta V(t) I(t)+\frac{mu_{2}(t)I(t)}{1+\eta I(t)} \right)dt\\
&-(1-e)\sigma V(t) I(t) dB_{3}(t).
\end{split}
\end{cases}
\end{equation}
List of parameters, variables, and their meanings in model.\\
$p$:	 The proportion of population that get  vaccinated immediately after birth \\
$u_1(t)$:  Vaccinated rate; $u_2(t)$:  Treatment rate \\
$\beta$: Per capita transmission coefficient  \\
$\mu$: Per capita natural death rate	/ birth rate \\
$\alpha$: Per capita recovery rate; $e$:  The degree of effectiveness of vaccination    \\
$\sigma$: The intensities of the white noise\\
$S(t)$:   Susceptible proportion in the total population \\
$I(t)$: Infected proportion in the total population \\
$V(t)$: Vaccinated proportion in the total population.

%As we know, ODE models are relatively simple and easy to deal with, but they ignore spatial heterogeneity and movements of individuals, and of course cannot exhibit spatial patterns of the diseases. In addition, the outbreak predictions and the basic reproduction number can have severe departures from reality. To study the impact of spatial distribution and movements of individuals on the spread of diseases, Webb analyzed in [33] the classical Kermack-McKendrick model with diffusion and homogeneous Neumann boundary conditions on an interval. For partial differential equation (PDE) models, if we consider the factors including birth, death, migration, delay, spatio-temporal movements of individuals, infection age, spatial heterogeneity, and so on, the dynamical behaviors become much more complicated [2, 10, 19, 27, 34].

As we know, the factors including birth, death, and so on in epidemic systems are spatially heterogeneous \cite{82}. To study the impact of spatial distribution and movements of individuals on the spread of diseases, we need to established partial differential equation (PDE) SIV epidemic models, it can exhibit spatial patterns of the diseases \cite{3aa1}.  Let Laplacian operator $\Delta =\frac{\partial^{2}}{\partial S^{2}}+\frac{\partial^{2}}{\partial I^{2}}+\frac{\partial^{2}}{\partial V^{2}}$, the stochastic SIV model with spatial diffusion can present as
\begin{equation}\label{1`q2}
\begin{cases}
\begin{split}
dS(x,t)&=\bigg(D_{1}(x)\Delta S+(1-p(x))b(x)+\alpha(x) I(x,t)-(\mu(x)+u_1(x,t))S(x,t)\\
&-\beta(x) S(x,t) I(x,t)\bigg)dt-\sigma(x) S(x,t) I(x,t)dB_{1}(t),\\
dI(x,t)&= \bigg(D_{2}(x)\Delta I+\beta(x) S(x,t) I(x,t)  +(1-e)\beta(x) V(x,t) \\ &I(x,t)-(\mu(x)+\alpha(x))I(x,t)-\frac{m(x)u_{2}(x,t)I(x,t)}{1+\eta(x) I(x,t)}\bigg)dt\\
&+\sigma(x) S(x,t) I(x,t) dB_{2}(t)+(1-e)\sigma(x) V(x,t) I(x,t) dB_{4}(t),\\
dV(x,t)&= \bigg(D_{3}(x)\Delta V+ p(x) b(x)+u_1(x,t) S(x,t)-\mu(x)V(x,t) \\ &-(1-e)\beta(x) V(x,t) I(x,t)+\frac{m(x)u_{2}(x,t)I(x,t)}{1+\eta(x) I(x,t)} \bigg)dt\\
&-(1-e)\sigma(x) V(x,t) I(x,t) dB_{3}(t),
\end{split}
\end{cases}
\end{equation}
with $(x,t)\in \Gamma_{T}=(0,T)\times \Gamma$. The boundary conditions
\begin{equation}\label{3q}
\frac{\partial S(x,t)}{\partial \nu}=\frac{\partial I(x,t)}{\partial \nu}=\frac{\partial V(x,t)}{\partial \nu}=0,~~(x,t)\in \sum=(0,T)\times \partial \Gamma,
\end{equation}
where $\nu$ is the unit normal vector on $\partial \Gamma$. The initial boundary conditions  are
\begin{equation}\label{3w}
S(0,x)=S_{0}(x)>0,~~I(0,x)=I_{0}(x)>0,~~V(0,x)=V_{0}(x)>0, x\in \Gamma.
\end{equation}
where $m$ is the cure rate, $\eta$ in the treatment function measures the extent of delayed treatment given to infected people. Let $U \subseteq \mathbb{R}$ be a bounded nonempty closed set. The control $u(x,t)\in U$ is called admissible, if it is an $\mathcal{F}_{t}$-adapted process with values in $U$. The set of all admissible controls is denoted by $\mathcal{U}_{ad}$. In this optimal problem, we assume a restriction on the control variable such that $0 \leq u_{1}(x,t) \leq 1$, because it is impossible for all susceptible individuals to be vaccinated at one time. $u_{2}(x,t) = 0$ represents no treatment, and $u_{2}(x,t) = 1$ represents totally effective treatment.

%In this paper, we introduce Markov switching into the above model, and obtain the following model
Additionally, due to the sudden change of temperature, virus and other factors, the SIV epidemic system may experience abrupt changes in their parameters \cite{mhjk}. Continuous time Markov chain is widely used to characterize this kind of environmental noise in epidemic models \cite{mark1,mark2}. Therefore, it is reasonable to consider the SIV epidemic model with Markovian switching
\begin{equation}\label{1`q1}
\begin{cases}
\begin{split}
dS(x,t)&=\bigg(D_{1}(x)\Delta S(x,t)+(1-p(x,\Lambda_{t}))b(x,\Lambda_{t})+\alpha(x,\Lambda_{t}) I(x,t)-(\mu(x,\Lambda_{t})\\
&+u_1(x,t))S(x,t)-\beta(x,\Lambda_{t}) S(x,t) I(x,t)\bigg)dt-\sigma(x,\Lambda_{t}) S(x,t) I(x,t)dB_{1}(t),\\
dI(x,t)&= \bigg(D_{2}(x)\Delta I(x,t)+\beta(x,\Lambda_{t}) S(x,t) I(x,t)  +(1-e)\beta(x,\Lambda_{t}) V(x,t)I(x,t) \\ &-(\mu(x,\Lambda_{t})+\alpha(x,\Lambda_{t}))I(x,t)-\frac{m(x,\Lambda_{t})u_{2}(x,t)I(x,t)}{1+\eta(x,\Lambda_{t}) I(x,t)}\bigg)dt\\
&+\sigma(x,\Lambda_{t}) S(x,t) I(x,t) dB_{2}(t)+(1-e)\sigma(x,\Lambda_{t}) V(x,t) I(x,t) dB_{4}(t),\\
dV(x,t)&= \bigg(D_{3}(x)\Delta V(x,t)+ p(x,\Lambda_{t}) b(x,\Lambda_{t})+u_1(x,t) S(x,t)-\mu(x,\Lambda_{t})V(x,t) \\
&-(1-e)\beta(x,\Lambda_{t}) V(x,t) I(x,t)+\frac{m(x,\Lambda_{t})u_{2}(x,t)I(x,t)}{1+\eta(x,\Lambda_{t}) I(x,t)} \bigg)dt\\
&-(1-e)\sigma(x,\Lambda_{t}) V(x,t) I(x,t) dB_{3}(t).
\end{split}
\end{cases}
\end{equation}
To simplify equations, we represent $X(x,t):=(S(x,t), I(x,t), V(x,t))^{\top}$, $u(x,t):=(u_{1}(x,t), u_{2}(x,t))^{\top}$.% and the model \eqref{1`q1} can be rewritten as
%\begin{equation}\label{1`q1}
%\begin{cases}
%\begin{split}
%dS(x,t)&=\bigg(D_{1}(x)\Delta S(x,t)+(1-p(x,\Lambda_{t}))b(x,\Lambda_{t})+\alpha(x,\Lambda_{t}) I(x,t)-(\mu(x,\Lambda_{t})+u_1(x,t))S(x,t)\\
%&-\beta(x,\Lambda_{t}) S(x,t) I(x,t)\bigg)dt-\sigma(x,\Lambda_{t}) S(x,t) I(x,t)dB_{1}(x,t)\\
%&\equiv f_{1}(X(x,t),u(x,t))dt-\sigma_{14}(X(x,t))dB(t), \\
%dI(x,t)&= \bigg(D_{2}(x)\Delta I(x,t)+\beta(x,\Lambda_{t}) S(x,t) I(x,t)  +(1-e)\beta(x,\Lambda_{t}) V(x,t) I(x,t)\\ &-(\mu(x,\Lambda_{t})+\alpha(x,\Lambda_{t}))I(x,t)-\frac{m(x,\Lambda_{t})u_{2}(x,t)I(x,t)}{1+\eta(x,\Lambda_{t}) I(x,t)}\bigg)dt\\
%&+\sigma(x,\Lambda_{t}) S(x,t) I(x,t) dB_{2}(x,t)+(1-e)\sigma(x,\Lambda_{t}) V(x,t) I(x,t) dB_{4}(x,t)\\
%&\equiv f_{2}(X(x,t),u(x,t))dt-\sigma_{24}(X(x,t))dB(t), \\
%dV(x,t)&= \bigg(D_{3}(x)\Delta V(x,t)+ p(x,\Lambda_{t}) b(x,\Lambda_{t})+u_1(x,t) S(x,t)-\mu(x,\Lambda_{t})V(x,t) \\ &-(1-e)\beta(x,\Lambda_{t}) V(x,t) I(x,t)+u_{1}(x,t) S(x,t)\\
%&+\frac{m(x,\Lambda_{t})u_{2}(x,t)I(x,t)}{1+\eta(x,\Lambda_{t}) I(x,t)} \bigg)dt-(1-e)\sigma(x,\Lambda_{t}) V(x,t) I(x,t) dB_{3}(x,t)\\
%&\equiv f_{3}(X(x,t),u(x,t))dt-\sigma_{34}(X(x,t))dB(t), \\
%S(0,x)=&s_{1}(x)>0,~~I(0,x)=i_{1}(x)>0,~~V(0,x)=v_{1}(x)>0, x\in \Gamma.
%\end{split}
%\end{cases}
%\end{equation}
~The objective function
\begin{equation}
J(0, X(x,t); u(x,t))= \mathbb{E}\bigg(  \int_0^T\int_{\Gamma} L(x,t,X(x,t), u(x,t))dxdt  +\int_{\Gamma}h(X(x,T))dx   \bigg),
\end{equation}
where $X(x,t)=\{  X(x,t): 0 \leq t\leq T  \}$ is the solution of model \eqref{1`q1} on the filtered space $(\Omega, \mathcal{F}, (\mathcal{F}_t)_{0 \leq t \leq T}, \mathbb{P})$. For any $u(x,t) \in \mathcal{U}_{ad}$, model \eqref{1`q1} has a unique $\mathcal{F}_t$-adapted solution $X(x,t)$ and $(X(x,t), u(x,t))$ is called an admissible pair.
The control problem is to find an admissible control which minimizes or nearly minimizes the objective function $J(0,x_{0};u(\cdot))$ over all $u(\cdot)\in \mathcal{U}_{ad}$.  The admissible control set is defined as
\begin{equation}
\mathcal{U}_{ad}[0,T]=\{u:[0,T]\times \Omega \to \mathbb{R}^{3}|~u \mbox{~is~} \mathcal{F}\mbox{-adapted, ~and~} \mathbb{E}\int_{0}^{T}\int_{\Omega}|u|^{2}dxds<\infty \}.
\end{equation}
The objective function $J(0,x_{0};u(\cdot))$ represent a function that $t=0, X(x,t)=x_{0}$  with control variable $u(\cdot)$. The value function is as follows:
\begin{equation*}
V(0,x_{0})=\min_{u(\cdot)\in \mathcal{U}_{ad}}J(0,x_{0};u(\cdot)).
\end{equation*}
Our goal is to minimize the total number of the infected and susceptible individuals by using minimal control efforts. Let the objective function in this paper is
\begin{equation}\label{s1}
\begin{split}
&J(0, X(x,t); u(x,t))\\
=&\int_{0} ^{T}\int_{\Gamma}\left(A_{1}S(x,t)+A_{2}I(x,t)+\frac{1}{2}(\tau_{1} u^{2}_{1}(x,t)+\tau_{2} u^{2}_{2}(x,t))\right) \emph{dxdt}+\int_{\Gamma}h(X(x,T))dx,
\end{split}
\end{equation}
here, the positive constants $\tau_{1}$, $\tau_{2}$ and $A_{1}, A_{2}$ represent weights.

\begin{remark}In this paper, we will work on a specific cost function of the form
\begin{equation*}
\begin{split}
L(x,t,X(x,t); u(x,t))&=A_{1}S(x,t)+A_{2}I(x,t)+\frac{1}{2}(\tau_{1} u^{2}_{1}(x,t)+\tau_{2} u^{2}_{2}(x,t)),\\
h(X(x,T))&=(0,I(x,T),0),
\end{split}
\end{equation*}
where $A_{1},A_{2}$ and $\tau_{1},\tau_{2}$ are the weighting factors, representing the cost per unit time of the components $S(x,t), I(x,t), u_{1}^{2}, u_{2}^{2}$, respectively. In particular, $A_{2}I(x,t)$ is the cost that infected individuals creates for the society due to lost working hours and standard medical care, not including the vaccinated $u_{1}$ and treatment $u_{2}$, while $\tau_{1}u_{1}^{2}$ is the cost of vaccination and $\tau_{2} u^{2}_{2}$ is the cost of treating infected individuals. We assume that the cost is proportional to the number of infected individuals, and the cost per each patient depends quadratically on the treatment effort $u_{2}$, which means that it is marginally increasing in $u_{2}$.

\end{remark}
\subsection{ Preliminaries}

%where $\mathcal{L}:=(0,T)\times \Gamma$, $\Gamma$ is a bounded domain in $R^{3}$  with smooth boundary $\partial \Gamma$, $t\in (0,T)$; $S(x,t)$ denotes the population density  {at the location $x$ at time $t$ }. $I(x,t)$ is the concentration of toxicant in the organism at time $t$ and in spatial position $x$. The concentration of toxicant in the environment at the location $x$ at time $t$  {is} described by $V(x,t)$. $K(\Lambda_{t})$ is the net organismal uptake rate of toxicant from the environment at time $t$. $M(\Lambda_{t})$ is the total loss rate of the toxicant from the environment. $\mu(t,x,I(x,t),\Lambda_{t})$ denotes the decreasing rate function of the population {at time $t$} and in spatial position $x$. $k_{i}>0, i=1,2,3$ is the diffusion coefficient.  $\beta(t,x,I(x,t),\Lambda_{t})$ describes the intrinsic growth rate function of the population at time $t$ and in spatial position $x$. $u(x,t)$ denotes the exogenous total toxicant input into environment at time $t$ and in spatial position $x$. $l(\Lambda_{t}) $ is the net organismal excretion rate of toxicant and $m(\Lambda_{t})$ is depuration rate of toxicant due to metabolic process and other losses.
Let $K=H^{1}(\Gamma)\equiv \{ \varphi|\varphi \in L^{2}(\Gamma), \frac{\partial \varphi}{\partial x_{i}}\in L^{2}(\Gamma), i=1,2,3.\}$ Then $K'=H^{-1}(\Gamma)$ is  the space of $K$. We denote by $|\cdot|$ and $\|\cdot\|$ the norms in $K$ and $K'$, respectively, by $\langle \cdot, \cdot \rangle$ the duality product between $K, K'$ and by $(\cdot, \cdot)$ the scalar product in $K$.  Let $(\Omega,\mathcal{F}, \mathbb{P})$ be a complete probability space with $\{\mathcal{F}_{t}\}_{0\leq t \leq T}$ the natural filtration generated by the  Brownian motion $B(t)$, which means $\mathcal{F}_{t}=\sigma\{B(t);0\leq s\leq t\}$ augmented with all $\mathbb{P}$-null sets of $\mathcal{F}_{0}$.~%$U \subset R$ denotes a bounded nonempty closed set, random process $\{u(t), t \geq 0\}$ is progressively measurable with respect to the natural filtration $(\mathcal{F}_{t})_{0 \leq t \leq T}$, taking values from $U$.
 %Throughout the paper, Let $(V,\|\cdot\|)$ and $(H,|\cdot|)$ be two separable Hilbert {spaces}, with norm denoted by {$\|\cdot\|$ and $|\cdot|$, respectively}. $V$ is viewed as a subspace of $H$ with a continuous dense embedding. $V\Subset H$ represents the embedding is compact. $V'$ and $H'$ are the dual of $V$, $H$. We set $H_{3}:= H\times H \times H$. Let $(\Omega,\mathcal{F},\mathbb{P})$ be a complete probability space with $\{\mathcal{F}_{t}\}_{0\leq t\leq T}$ the natural filtration generated by the Brownian motion $W_{t}$, which means $\mathcal{F}_{t}=\sigma\{W_{s};0\leq s\leq t\}$ augmented with all {$\mathbb{P}$-null} sets of $\mathcal{F}_{0}$.  {To construct such a filtration, we denote by $\mathcal{N}$ the collection of $\mathbb{P}$-null sets, i.e. $\mathcal{N}=\{B\in\mathcal{F}:\mathbb{P}(B)=0\}$.
In the paper, $C>0$ represents different positive constants. $\Lambda_{t},$ $t>0$, be a right-continuous Markov chain on the probability space taking values in a finite state $\mathbb{S}=\{1,2,\ldots,N\}$ for some positive integer $N<\infty$.
In order to prove the existence and uniqueness of  the underlying invariant measure, we prepare the following lemma.
\begin{lemma}\label{lko085}$($\cite{3}$)$
Let $ N<\infty $ and assume further that $\sum_{i=1}^{N}\mu_{i}\rho_{i}<0,$
where $\mu_{i}$ is the stationary distribution of Markov chain $\{\Lambda_{t}\}_{t\geq 0}$. Then $(1)$ $\eta_{p}>0$ if $\max\limits_{i\in \mathbb{S}}\rho_{i}\leq0$;~
$(2)$ $\eta_{p}>0$ for $p<\max\limits_{i\in \mathbb{S},\rho_{i}>0}\{-2q_{ii}/\rho_{i}\}$ if $\max\limits_{i\in \mathbb{S}}\rho_{i}>0$.
\end{lemma}
{The generator of $\{\Lambda_{t}\}_{t>0}$ is specified by $Q=(q_{ij})_{N\times N}$, such that for a sufficiently small $\Delta$,}\\
\begin{equation}\label{bgth876}
\mathbb P(\Lambda_{t+\Delta}=j|\Lambda_{t}=i)=
\begin{cases}
q_{ij}\Delta+o(\Delta), &\quad i\neq j,\\
1+q_{ii}\Delta+o(\Delta), &\quad i=j,
\end{cases}
\end{equation}
where $\Delta >0$, $o(\Delta)$ satisfies $\lim_{\Delta\rightarrow 0}\frac{o(\Delta)}{\Delta}=0$. {Here $q_{ij}$ is the transition rate from $i$ to $j$ satisfying $q_{ii}=-\sum\limits_{i\neq j}q_{ij}$.}  We assume that the Markov chain $\{\Lambda_{t}\}$ defined on the probability space above is independent of the standard Brownian motion $\{B(t)\}_{t\geq 0}$ and the {$Q$ matrix} is irreducible and conservative. Therefore, the Markov chain $\{\Lambda_{t}\}_{t\geq 0}$ has a unique stationary distribution $\pi:= (\pi_{1},\ldots,\pi_{N})$ which can be determined by solving the linear equation
\begin{equation*}
\pi Q=\bf{0} \quad \quad {\mbox{subject~to}}\quad \sum_{i=1}^{N}\pi_{i}=1 ~ \mbox{with}~\pi_{i}>0.
\end{equation*}
 Let $\mathcal{P}(H_{3}\times\mathbb{S})$ stand for the family of all probability measures on $H_{3}\times\mathbb{S}$. For $\xi=(\xi_{1},\xi_{2},\xi_{3})^{\ast} \in H_{3}$, $\xi\gg \bf{0}$ means each component $\xi_{i}>0$ , $i=1,2,3$.

We replace $((S(x,t), I(x,t), V(x,t)),\Lambda_{t})$ with $((S(x,t)^{s_{1},i}, I(x,t)^{i_{1},i}, V(x,t)^{v_{1},i}),\Lambda_{t}^{i})$,
especially the initial value $$((S^{0},I^{0},V^{0}), \Lambda_{0})=((s_{1},i_{1},v_{1}),i).$$ For any $p\in (0,1]$, we set $s:=(s_{1},i_{1},v_{1})$ and  {define} a metric on $ H_{3}\times\mathbb{S}$ as follows
\begin{equation*}
\begin{split}
d_{p}((s,i),(\bar{s},i)):=\int_{H_{3}}\sum_{k=1}^{3}|s_{k}-\bar{s}_{k}|^{p}+\chi_{\{i\neq j\}}, \quad (s,i),(\bar{s},i)\in H_{3}\times\mathbb{S},
\end{split}
\end{equation*}
where $\chi_{A}$ denotes the indicator function of the set $A$, and $\bar{s}:=(\bar{s}_{1}, \bar{i}_{1}, \bar{v}_{1})$ is different initial  {value}. For $p\in (0,1]$, we define the Wassertein distance between $\nu\in\mathcal{P}(H_{3}\times\mathbb{S})$ and $\nu'\in\mathcal{P}(H_{3}\times\mathbb{S})$ by
$$W_{p}(\nu,\nu')=\inf\mathbb Ed_{p}(X_{k},X_{k'}),$$ where the infimum is taken over all pairs of random variables $X_{k}$, $X_{k'}$ on $ H_{3} \times\mathbb{S}$ with respective laws $\nu$, $\nu'$. Let $\mathbb{P}_{t}((s_{1},i_{1},v_{1}),i;\cdot)$ be the transition probability kernel of the pair $((S(x,t)^{s_{1},i}, I(x,t)^{i_{1},i}, V(x,t)^{v_{1},i}),\Lambda_{t}^{i})$, a time homogeneous Markov process (see \cite{3}). Recall that $\pi\in \mathcal{P}(H_{3}\times\mathbb{S})$ is called an invariant measure of $((S(x,t)^{s_{1},i}, I(x,t)^{i_{1},i}, V(x,t)^{v_{1},i}),\Lambda_{t}^{i})$ if
\begin{equation}\label{123c}
\pi(A\times \{i\})= \sum_{j=1}^{N}\int_{H_{3}} \mathbb{P}_{t}((s_{1},i_{1},v_{1}),j; A\times \{i\})\pi(d(s_{1},i_{1},v_{1}) \times \{j\}), t\geq 0,  A \in H_{3}, i\in \mathbb{S},
\end{equation}
holds. For any $p>0$, let
\begin{equation}\label{567mos}
{\mbox{diag}}(\rho)\triangleq {\mbox{diag}}(\rho_{1},\ldots,\rho_{N}), \quad Q_{p}\triangleq Q+\frac{p}{2}{\mbox{diag}}(\rho),\quad \eta_{p}\triangleq-\max\limits_{\gamma}Re\gamma,
\end{equation}
where $\rho_{i}$ is introduced in the  {assumptions} and $\gamma \in \mbox{spec}(Q_{p})$, $\mbox{spec}(Q_{p})$ denotes the spectrum of $Q_{p}$(i.e., the multi-set of its eigenvalues). $Re\gamma$ is the real part of $\gamma$ and  ${\mbox{diag}}(\rho_{1},\ldots,\rho_{N})$ denotes  {the diagonal matrix whose diagonal entries are }$\rho_{1},\ldots,\rho_{N}$, respectively.

\section{Existence and uniqueness of invariant measures}\label{Section3}
One of the most important problems in applying stochastic analysis is the question of the existence and uniqueness of invariant measures for specific processes. In this section, we mainly prove the existence and uniqueness of the invariant measure for the exact solution. Thus one can expect that processes with invariant measures exhibit some kind of stability. If \autoref{lko085} hold, the system (\ref{1`q1}) is said to be attractive ``in average ".
\begin{theorem}\label{mmy6}
Let $N< \infty$ and $\max_{i\in S}\rho_{i}>0$, then
the exact solution of system \eqref{1`q1} admits a unique invariant measure $\pi\in \mathcal{P}(H_{3}\times\mathbb{S})$.
\end{theorem}
\begin{proof}
The key point of proof is to divide the whole proof into two parts of existence and uniqueness.

$(\RNum{1})$   {Existence} of invariant measure.  Let $((S^{s_{1},i}_{t},I^{i_{1},i}_{t},V^{v_{1},i}_{t}),\Lambda_{t}^{i})$ be the exact solution of system \eqref{1`q1} with $((s_{1},i_{1},v_{1}),i)$ as initial value, where $((s_{1},i_{1},v_{1}),i)\in H_{3}\times \mathbb{S}$.
 A simple application of the Feynman-Kac formula show that
 let $Q_{p,t}=e^{tQ_{p}}$, where $Q_{p}$ is given in \eqref{567mos}. Then, the spectral radius Ria$(Q_{p,t})$ (i.e.,Ria$(Q_{p,t})=\sup_{\lambda\in \mbox{spec}(Q_{p,t})}|\lambda|$) of $Q_{p,t}$  equals to $e^{-\eta_{p}t}.$ Since all {coefficients} of $Q_{p,t}$ are positive, by the Perron-Frobenius theorem  yields that $-\eta_{p}$ is a simple eigenvalue of $Q_{p}$, all other eigenvalues  {have} a strictly smaller real part. Note that the eigenvector of $Q_{p,t}$ corresponding to $e^{-\eta_{p}t}$ is also an eigenvector of $Q_{p}$ corresponding to $-\eta_{p}$. According to Perron-Frobenius theorem, for $Q_{p}$ it can be found that there is a positive eigenvector $\xi^{(p)}=(\xi_{1}^{(p)},\ldots,\xi_{N}^{(p)})\gg \bf{0} $ (where $\bf{0}\in H$ is a zero vector) corresponding to the eigenvalue $-\eta_{p}$, and  $\xi^{(p)}\gg \bf{0} $ means that each component $\xi_{i}^{(p)}>0$. Let
\begin{equation}\label{78y5r}
p_{0}=1\wedge \min\limits_{i\in\mathbb{S},\rho_{i}>0}\{-2q_{ii}/\rho_{i}\},
\end{equation}
 {where $1\wedge \min\limits_{i\in\mathbb{S},\rho_{i}>0}\{-2q_{ii}/\rho_{i}\}:=\min\{1,\min\limits_{i\in\mathbb{S},\rho_{i}>0}\{-2q_{ii}/\rho_{i}\}\}$}. Combined with \autoref{lko085}, we can get
\begin{equation}\label{12345frg}
Q_{p}\xi_{i}^{(p)}=-\eta_{p}\xi_{i}^{(p)}\ll \bf{0}.
\end{equation}
 {In order to investigate the existence and uniqueness of invariant measure for exact solution, we need to prove the boundedness of exact solution for system (\ref{1`q1}).} In other words, we need to prove whether the following inequality holds
$$\mathbb{E}(1+|S_{t}^{s_{1},i}|^{p}+|I_{t}^{i_{1},i}|^{p}+|V_{t}^{v_{1},i}|^{p})\leq C.$$
First, using the It$\hat{\mbox{o}}$ formula, we have
\begin{equation*}
\begin{split}
&e^{\eta_{p}t}\mathbb{E}((1+|S_{t}^{s_{1},i}|^{2}+|I_{t}^{i_{1},i}|^{2}
+|V_{t}^{v_{1},i}|^{2})^{p/2}\xi_{\Lambda_{t}^{i}}^{(p)})\\
=&(1+|s_{1}|^{2}+|i_{1}|^{2}+|v_{1}|^{2})^{\frac{p}{2}}\xi_{i}^{p}+\mathbb{E}
\int_{0}^{t}e^{\eta_{p}\epsilon}(1+|S_{\epsilon}^{s_{1},i}|^{2}
+|I_{\epsilon}^{i_{1},i}|^{2}+|V_{\epsilon}^{v_{1},i}|^{2})^{\frac{p}{2}}
\bigg\{\eta_{p}\xi_{\Lambda_{\epsilon}^{i}}^{(p)}
+(Q\xi^{(p)})(\Lambda_{\epsilon}^{i})\bigg\}d\epsilon\\
&+\frac{p}{2}\mathbb{E}\int_{0}^{t}e^{\eta_{p}\epsilon}(1+|S_{\epsilon}^{s_{1},i}|^{2}+|I_{\epsilon}^{i_{1},i}|^{2}
+|V_{\epsilon}^{v_{1},i}|^{2})^{\frac{p}{2}-1}\xi_{\Lambda_{\epsilon}^{i}}^{(p)}\bigg\{2\langle S_{\epsilon}^{s_{1},i}, D_{1}(x)\Delta S_{\epsilon}^{s_{1},i}+(1-p(\Lambda_{\epsilon}^{i}))b(\Lambda_{\epsilon}^{i})
+\alpha(\Lambda_{\epsilon}^{i}) I_{\epsilon}^{i_{1},i}\\
&-(\mu(\Lambda_{\epsilon}^{i})+u_1(x,t))S_{\epsilon}^{s_{1},i}-\beta(\Lambda_{\epsilon}^{i}) S_{\epsilon}^{s_{1},i} I_{\epsilon}^{i_{1},i}\rangle+
2 \langle I_{\epsilon}^{i_{1},i}, D_{2}(x)\Delta I_{\epsilon}^{i_{1},i}+\beta(\Lambda_{\epsilon}^{i}) S_{\epsilon}^{s_{1},i} I_{\epsilon}^{i_{1},i}  +(1-e)\beta(\Lambda_{\epsilon}^{i}) V_{\epsilon}^{v_{1},i} I_{\epsilon}^{i_{1},i}\\
&-(\mu(\Lambda_{\epsilon}^{i})+\alpha(\Lambda_{\epsilon}^{i}))I_{\epsilon}^{i_{1},i}
-\frac{m(\Lambda_{\epsilon}^{i})u_{2}(x,t)I_{\epsilon}^{i_{1},i}}{1+\eta(\Lambda_{\epsilon}^{i}) I_{\epsilon}^{i_{1},i}}\rangle+
2 \langle V_{\epsilon}^{v_{1},i}, D_{3}(x)\Delta V_{\epsilon}^{v_{1},i}+ p (\Lambda_{\epsilon}^{i})b(\Lambda_{t})+u_1(x,t) S_{\epsilon}^{s_{1},i}-\mu(\Lambda_{\epsilon}^{i}) V_{\epsilon}^{v_{1},i} \\
&-(1-e)\beta(\Lambda_{\epsilon}^{i}) V_{\epsilon}^{v_{1},i} I_{\epsilon}^{i_{1},i}+u_{1}(x,t) S_{\epsilon}^{s_{1},i}+\frac{m(\Lambda_{\epsilon}^{i})u_{2}(x,t)I_{\epsilon}^{i_{1},i}}{1
+\eta(\Lambda_{\epsilon}^{i}) I_{\epsilon}^{i_{1},i}} \rangle \bigg\}d\epsilon+\frac{p}{2}\int_{0}^{t}e^{\eta_{p}\epsilon}\mathbb{E}(1+|S_{\epsilon}^{s_{1},i}|^{2}
+|I_{\epsilon}^{i_{1},i}|^{2}+|V_{\epsilon}^{v_{1},i}|^{2})^{\frac{p}{2}-1}\\
&\times\xi_{\Lambda_{\epsilon}^{i}}^{(p)}\bigg\{(p-2)(1+|S_{\epsilon}^{s_{1},i}|^{2}
+|I_{\epsilon}^{i_{1},i}|^{2}+|V_{\epsilon}^{v_{1},i}|^{2})^{-1}
\bigg(\|S_{\epsilon}^{s_{1},i} \sigma(\Lambda_{\epsilon}^{i}) S_{\epsilon}^{s_{1},i} I_{\epsilon}^{i_{1},i}\|^{2}+\|I_{\epsilon}^{i_{1},i}(\sigma(\Lambda_{\epsilon}^{i}) S_{\epsilon}^{s_{1},i} I_{\epsilon}^{i_{1},i} \\
&+(1-e)\sigma(\Lambda_{\epsilon}^{i}) V_{\epsilon}^{v_{1},i} I_{\epsilon}^{i_{1},i} \|^{2}
+\|V_{\epsilon}^{v_{1},i}(1-e)\sigma(\Lambda_{\epsilon}^{i}) V_{\epsilon}^{v_{1},i} I_{\epsilon}^{i_{1},i} \|^{2}\bigg)\\
&+\bigg(\| \sigma(\Lambda_{t}) S_{\epsilon}^{s_{1},i} I_{\epsilon}^{i_{1},i}\|^{2}+\|(\sigma(\Lambda_{t}) S_{\epsilon}^{s_{1},i} I_{\epsilon}^{i_{1},i} +(1-e)\sigma(\Lambda_{\epsilon}^{i}) V_{\epsilon}^{v_{1},i} I_{\epsilon}^{i_{1},i} \|^{2}
+\|(1-e)\sigma(\Lambda_{\epsilon}^{i}) V_{\epsilon}^{v_{1},i} I_{\epsilon}^{i_{1},i} \|^{2}\bigg)\bigg\}d\epsilon.
\end{split}
\end{equation*}
Using $p(p-2)/2<0$, due to $p\in (0,p_{0}),$ and combining with the following inequality,
\begin{equation}\label{234cv}
\begin{split}
\int_{0}^{t}\int_{\Omega}D_{1}(x)\Delta S_{\epsilon} S_{\epsilon}dxd\epsilon=-\int_{0}^{t}\int_{\Omega}D_{1}(x)\triangledown S_{\epsilon}\triangledown S_{\epsilon}dxd\epsilon \leq -D_{0}\int_{0}^{t}\|S_{\epsilon}\|^{2}d\epsilon,
\end{split}
\end{equation}
by similar methods, we have
\begin{equation}\label{234cv}
\begin{split}
&\int_{0}^{t}\int_{\Omega}D_{2}(x)\Delta I_{\epsilon} I_{\epsilon}dxd\epsilon
\leq -D_{0}\int_{0}^{t}\|I_{\epsilon}\|^{2}d\epsilon,\\
&\int_{0}^{t}\int_{\Omega}D_{3}(x)\Delta V_{\epsilon} V_{\epsilon}dxd\epsilon
\leq -D_{0}\int_{0}^{t}\|V_{\epsilon}\|^{2}d\epsilon,
\end{split}
\end{equation}
where $0\leq D_{0}\leq D_{i}(x,t)<\infty,~i=1,2,3,$ ($D_{0}$ is a constant).\\
Further, we have
\begin{equation*}
\begin{split}
&e^{\eta_{p}t}\mathbb{E}((1+|S_{t}^{s_{1},i}|^{2}+|I_{t}^{i_{1},i}|^{2}
+|V_{t}^{v_{1},i}|^{2})^{p/2}\xi_{\Lambda_{t}^{i}}^{(p)})\\
\leq&(1+|s_{1}|^{2}+|i_{1}|^{2}+|v_{1}|^{2})^{\frac{p}{2}}\xi_{i}^{p}
+\mathbb{E}\int_{0}^{t}e^{\eta_{p}\epsilon}(1+|S_{\epsilon}^{s_{1},i}|^{2}
+|I_{\epsilon}^{i_{1},i}|^{2}
+|V_{\epsilon}^{v_{1},i}|^{2})^{\frac{p}{2}}\bigg\{\eta_{p}\xi_{\Lambda_{\epsilon}^{i}}^{(p)}
+(Q\xi^{(p)})(\Lambda_{\epsilon}^{i})\bigg\}d\epsilon\\
\end{split}
\end{equation*}
\begin{equation*}
\begin{split}
&+\frac{p}{2}\mathbb{E}\int_{0}^{t}e^{\eta_{p}\epsilon}(1
+|S_{\epsilon}^{s_{1},i}|^{2}+|I_{\epsilon}^{i_{1},i}|^{2}
+|V_{\epsilon}^{v_{1},i}|^{2})^{\frac{p}{2}-1}\xi_{\Lambda_{\epsilon}^{i}}^{(p)}\bigg\{2\langle S_{\epsilon}^{s_{1},i},  (1-p(\Lambda_{\epsilon}^{i}))b(\Lambda_{\epsilon}^{i})+\alpha (\Lambda_{\epsilon}^{i})I_{\epsilon}^{s_{1},i}\\
&-(\mu(\Lambda_{\epsilon}^{i})+u_1(x,t))S_{\epsilon}^{s_{1},i}-\beta (\Lambda_{\epsilon}^{i})S_{\epsilon}^{s_{1},i} I_{\epsilon}^{i_{1},i}\rangle+
2 \langle I_{\epsilon}^{i_{1},i}, \beta(\Lambda_{\epsilon}^{i}) S_{\epsilon}^{s_{1},i} I_{\epsilon}^{i_{1},i}  +(1-e)\beta(\Lambda_{\epsilon}^{i}) V_{\epsilon}^{v_{1},i} I_{\epsilon}^{i_{1},i}\\
&-(\mu(\Lambda_{\epsilon}^{i})
+\alpha(\Lambda_{\epsilon}^{i}))I_{\epsilon}^{i_{1},i}-\frac{m(\Lambda_{\epsilon}^{i})u_{2}(x,t)I_{\epsilon}^{i_{1},i}}{1
+\eta(\Lambda_{\epsilon}^{i}) I_{\epsilon}^{i_{1},i}}\rangle+
2 \langle V_{\epsilon}^{v_{1},i},  p(\Lambda_{\epsilon}^{i}) b(\Lambda_{\epsilon}^{i})+u_1(x,t) S_{\epsilon}^{s_{1},i}-\mu(\Lambda_{\epsilon}^{i}) V_{\epsilon}^{v_{1},i} \\
&-(1-e)\beta(\Lambda_{\epsilon}^{i}) V_{\epsilon}^{v_{1},i} I_{\epsilon}^{i_{1},i}+\frac{m(\Lambda_{\epsilon}^{i})u_{2}(x,t)I_{\epsilon}^{i_{1},i}}{1
+\eta(\Lambda_{\epsilon}^{i}) I_{\epsilon}^{i_{1},i}} \rangle \bigg\}d\epsilon+\frac{p}{2}\int_{0}^{t}e^{\eta_{p}\epsilon}\mathbb{E}(1
+|S_{\epsilon}^{s_{1},i}|^{2}+|I_{\epsilon}^{i_{1},i}|^{2}
\\
&+|V_{\epsilon}^{v_{1},i}|^{2})^{\frac{p}{2}-1}\times\xi_{\Lambda_{\epsilon}^{i}}^{(p)}\bigg\{(p-2)(1
+|S_{\epsilon}^{s_{1},i}|^{2}+|I_{\epsilon}^{i_{1},i}|^{2}
+|V_{\epsilon}^{v_{1},i}|^{2})^{-1}
\bigg(\|S_{\epsilon}^{s_{1},i} \sigma(\Lambda_{\epsilon}^{i}) S_{\epsilon}^{s_{1},i} I_{\epsilon}^{i_{1},i}\|^{2} \\
&+\|I_{\epsilon}^{i_{1},i}(\sigma S_{\epsilon}^{s_{1},i} I_{\epsilon}^{i_{1},i}+(1-e)\sigma(\Lambda_{\epsilon}^{i}) V_{\epsilon}^{v_{1},i} I_{\epsilon}^{i_{1},i} \|^{2}
+\|V_{\epsilon}^{v_{1},i}(1-e)\sigma(\Lambda_{\epsilon}^{i}) V_{\epsilon}^{v_{1},i} I_{\epsilon}^{i_{1},i} \|^{2}\bigg)\\
&+\bigg(\| \sigma(\Lambda_{\epsilon}^{i}) S_{\epsilon}^{s_{1},i} I_{\epsilon}^{i_{1},i}\|^{2}+\|(\sigma (\Lambda_{\epsilon}^{i})S_{\epsilon}^{s_{1},i} I_{\epsilon}^{i_{1},i} +(1-e)\sigma(\Lambda_{\epsilon}^{i}) V_{\epsilon}^{v_{1},i} I_{\epsilon}^{i_{1},i} \|^{2}
+\|(1-e)\sigma(\Lambda_{\epsilon}^{i}) V_{\epsilon}^{v_{1},i} I_{\epsilon}^{i_{1},i} \|^{2}\bigg)\bigg\}d\epsilon.
\end{split}
\end{equation*}
Therefore, based on the inequality $2ab\leq \varepsilon a^{2}+\frac{1}{\varepsilon}b^{2}$, $\varepsilon>0$, we can obtain
\begin{small}
\begin{equation*}
\begin{split}
&e^{\eta_{p}t}\mathbb{E}((1+|S_{t}^{s_{1},i}|^{2}+|I_{t}^{i_{1},i}|^{2}+|V_{t}^{v_{1},i}|^{2})^{p/2}\xi_{\Lambda_{t}^{i}}^{(p)})\\
\leq&(1+|s_{1}|^{2}+|i_{1}|^{2}+|v_{1}|^{2})^{\frac{p}{2}}\xi_{i}^{(p)}+\frac{p}{2}\mathbb{E}\int_{0}^{t}e^{\eta_{p}\epsilon}(1+|S_{\epsilon}^{s_{1},i}|^{2}+|I_{\epsilon}^{i_{1},i}|^{2}+|V_{\epsilon}^{v_{1},i}|^{2})^{\frac{p}{2}-1}\bigg\{(\epsilon \alpha(\Lambda_{\epsilon}^{i})+\beta(\Lambda_{\epsilon}^{i}) \epsilon \\
&+\frac{1}{\epsilon}-2(\mu(\Lambda_{\epsilon}^{i})+u_{1}))|S_{\epsilon}^{s_{1},i}|^{2}+(\frac{1}{\epsilon}+\beta(\Lambda_{\epsilon}^{i}) \epsilon+(1-e)\beta(\Lambda_{\epsilon}^{i}) \epsilon+\frac{1}{\epsilon(1+\eta(\Lambda_{\epsilon}^{i}) I_{\epsilon}^{i_{1},i})}\\
&-2(\mu(\Lambda_{\epsilon}^{i})+\alpha(\Lambda_{\epsilon}^{i})))|I_{\epsilon}^{i_{1},i}|^{2}+(u_{1}\epsilon+(1-e)\beta(\Lambda_{\epsilon}^{i}) \epsilon +u_{1}\epsilon+m(\Lambda_{\epsilon}^{i}) u_{2}\epsilon-2\mu(\Lambda_{\epsilon}^{i}))|V_{\epsilon}^{v_{1},i}|^{2}+\frac{2}{\epsilon}|S_{\epsilon}^{s_{1},i}I_{\epsilon}^{i_{1},i}|^{2}\\
&+\frac{2}{\epsilon}|I_{\epsilon}^{i_{1},i}V_{\epsilon}^{v_{1},i}|^{2}\bigg\}\xi_{\Lambda_{\epsilon}^{i}}^{(p)}d\epsilon+\frac{p}{2}\int_{0}^{t}e^{\eta_{p}\epsilon}\mathbb{E}(1+|S_{\epsilon}^{s_{1},i}|^{2}+|I_{\epsilon}^{i_{1},i}|^{2}+|V_{\epsilon}^{v_{1},i}|^{2})^{\frac{p}{2}-1}\bigg( \| \sigma S_{\epsilon}^{s_{1},i} I_{\epsilon}^{i_{1},i}\|^{2} \\
&+\|(\sigma(\Lambda_{\epsilon}^{i}) S_{\epsilon}^{s_{1},i} I_{\epsilon}^{i_{1},i}+(1-e)\sigma(\Lambda_{\epsilon}^{i}) V_{\epsilon}^{v_{1},i} I_{\epsilon}^{i_{1},i} \|^{2}+\|(1-e)\sigma (\Lambda_{\epsilon}^{i})V_{\epsilon}^{v_{1},i} I_{\epsilon}^{i_{1},i} \|^{2}\bigg)\xi_{\Lambda_{\epsilon}^{i}}^{(p)}d\epsilon\\
&+\mathbb{E}\int_{0}^{t}e^{\eta_{p}\epsilon}(1+|S_{\epsilon}^{s_{1},i}|^{2}+|I_{\epsilon}^{i_{1},i}|^{2}+|V_{\epsilon}^{v_{1},i}|^{2})^{\frac{p}{2}}\bigg\{\eta_{p}\xi_{\Lambda_{\epsilon}^{i}}^{(p)}+(Q\xi^{(P)})(\Lambda_{\epsilon}^{i})\bigg\}d\epsilon.
\end{split}
\end{equation*}
\end{small}
According to $(|a|+|b|)^{r}\leq 2^{r-1}(|a|^{r}+|b|^{r}),~r\geq1,$~$\forall a,b\in \mathbb{R},$ we can further estimate
\begin{equation*}
\begin{split}
&e^{\eta_{p}t}\mathbb{E}((1+|S_{t}^{s_{1},i}|^{2}+|I_{t}^{i_{1},i}|^{2}+|V_{t}^{v_{1},i}|^{2})^{p/2}\xi_{\Lambda_{t}^{i}}^{(p)})\\
\leq&c(1+|s_{1}|^{p}+|i_{1}|^{p}+|v_{1}|^{p})+\mathbb{E}\int_{0}^{t}e^{\eta_{p}\epsilon}(1+|S_{\epsilon}^{s_{1},i}|^{2}+|I_{\epsilon}^{i_{1},i}|^{2}+|V_{\epsilon}^{v_{1},i}|^{2})^{\frac{p}{2}}\bigg\{\eta_{p}\xi_{\Lambda_{s}^{i}}^{(p)}+(Q\xi^{(P)})(\Lambda_{\epsilon}^{i})\bigg\}d\epsilon\\
&+\frac{p}{2}\mathbb{E}\int_{0}^{t}e^{\eta_{p}\epsilon}(1+|S_{\epsilon}^{s_{1},i}|^{2}+|I_{\epsilon}^{i_{1},i}|^{2}+|V_{\epsilon}^{v_{1},i}|^{2})^{\frac{p}{2}}\bigg\{\frac{C_{1}|S_{\epsilon}^{s_{1},i}|^{2}+\tau_{1}|I_{\epsilon}^{i_{1},i}|^{2}}{(1+|S_{\epsilon}^{s_{1},i}|^{2}+|I_{\epsilon}^{i_{1},i}|^{2}+|V_{\epsilon}^{v_{1},i}|^{2})}\bigg\}d\epsilon\\
&+\frac{p}{2}\mathbb{E}\int_{0}^{t}e^{\eta_{p}\epsilon}(1+|S_{\epsilon}^{s_{1},i}|^{2}+|I_{\epsilon}^{i_{1},i}|^{2}+|V_{\epsilon}^{v_{1},i}|^{2})^{\frac{p}{2}}\bigg\{\frac{\tau_{2}|V_{\epsilon}^{v_{1},i}|^{2}+C_{4}|u_{\epsilon}|^{2}}{(1+|S_{\epsilon}^{s_{1},i}|^{2}+|I_{\epsilon}^{i_{1},i}|^{2}+|V_{\epsilon}^{v_{1},i}|^{2})}\bigg\}d\epsilon.
\end{split}
\end{equation*}
where $C_{1}:=(\epsilon \alpha(\Lambda_{\epsilon}^{i})+\beta(\Lambda_{\epsilon}^{i}) \epsilon +\frac{1}{\epsilon}-2(\mu(\Lambda_{\epsilon}^{i})+u_{1}))$,
$\tau_{1}:=(\frac{1}{\epsilon}+\beta(\Lambda_{\epsilon}^{i}) \epsilon+(1-e)\beta(\Lambda_{\epsilon}^{i}) \epsilon+\frac{1}{\epsilon(1+\eta(\Lambda_{\epsilon}^{i}) I_{\epsilon}^{s_{1},i})}-2(\mu(\Lambda_{\epsilon}^{i})+\alpha(\Lambda_{\epsilon}^{i})))$ and
$\tau_{2}:=(u_{1}\epsilon+(1-e)\beta(\Lambda_{\epsilon}^{i}) \epsilon +u_{1}\epsilon+m(\Lambda_{\epsilon}^{i}) u_{2}\epsilon-2\mu)$, $C_{4}:=\varepsilon_{2}+c$ are different constants.
Finally, by  {Gronwall's lemma}, we can get the result
\begin{equation}\label{bnmp}
\begin{split}
e^{\eta_{p}t}\mathbb{E}((1+|S_{t}^{s_{1},i}|^{2}+|I_{t}^{i_{1},i}|^{2}+|V_{t}^{v_{1},i}|^{2})^{p/2}\xi_{\Lambda_{t}^{i}}^{(p)})
\leq Ce^{CT},
\end{split}
\end{equation}
and further estimates can be obtained as follows
\begin{equation}\label{fgg564}
\sup\limits_{t\geq 0}\mathbb{E}((|S_{t}^{s_{1},i}|^{p}+|I_{t}^{i_{1},i}|^{p}+|V_{t}^{v_{1},i}|^{p})\leq C.
\end{equation}
For any $t>0$, we can define a probability measure
 $$\chi_{t}(A)=\frac{1}{t}\int_{0}^{t}\mathbb{P}_{\epsilon}(s,i;A)d\epsilon, \quad A\in (H_{3}\times\mathbb{S}).$$
 Then, let $Y_{t}^{s,i}:=(S_{t}^{s_{1},i}, I_{t}^{i_{1},i}, V_{t}^{v_{1},i})$, for any $\varepsilon>0$, by Eq.(\ref{fgg564}) and {Chebyshev's} inequality, there exists an $r>0$ sufficiently large such that
\begin{equation}\label{dfwg57}
\chi_{t}(K_{r}\times\mathbb{S} )=\frac{1}{t}\int_{0}^{t}\mathbb{P}_{\epsilon}(s,i;K_{r}\times\mathbb{S})d\epsilon\geq 1-\frac{\sup_{t\geq0}(\mathbb{E}|Y_{t}^{s,i}|^{p})}{r^{p}}\geq1-\varepsilon.
\end{equation}
Hence, $\chi_{t}$ is tight since the compact embedding $V\Subset H$, then $K_{r}=\{s\in H_{3}; |s|\leq r\}  $ is a compact subset of $H_{3}$  for each $i\in \mathbb{S}$. Combined with the Fellerian property of transition seimgroup for $\mathbb{P}_{t}(s,i;\cdot)$ and according to Krylov-Bogoliubov theorem (see \cite{4}) , $((S_{t}^{s_{1},i}, I_{t}^{i_{1},i}, V_{t}^{v_{1},i}),\Lambda_{t}^{i})$ has an invariant measure. Next, we prove the uniqueness of the invariant measure for $((S_{t}^{s_{1},i}, I_{t}^{i_{1},i}, V_{t}^{v_{1},i}),\Lambda_{t}^{i})$.

$(\RNum{2})$  Uniqueness of invariant measure.  First, let $((S_{t}^{s_{1},i}, I_{t}^{i_{1},i}, V_{t}^{v_{1},i}),\Lambda_{t}^{i})$ and $((S_{t}^{\bar{s}_{1},i}, I_{t}^{\bar{i}_{1},i}, V_{t}^{\bar{v}_{1},i}),\Lambda_{t}^{i})$ be the solutions of the system \eqref{1`q1} satisfying the initial values $((s_{1},i_{1},v_{1}),i)$ and $((\bar{s}_{1},\bar{i}_{1}, \bar{v}_{1}),i)$, respectively.  Under assumption conditions $(\mathbb{H}1)$-$(\mathbb{H}2)$, we take $\forall\varepsilon \in (0,1)$ and use It$\hat{\mbox{o}}$'s formula,  combined with Eq.(\ref{234cv}), it follows that
\begin{equation*}
\begin{split}
&e^{\eta_{p}t}\mathbb{E}((\varepsilon+|S_{t}^{s_{1},i}-S_{t}^{\bar{s}_{1},i}|^{2}+|I_{t}^{i_{1},i}
-I_{t}^{\bar{i}_{1},i}|^{2}+|V_{t}^{v_{1},i}
-V_{t}^{\bar{v}_{1},i}|^{2})^{p/2}\xi^{(p)}_{\Lambda_{t}^{i}})\\
\leq &(\varepsilon+|s_{1}-\bar{s}_{1}|^{2}+|i_{1}-\bar{i}_{1}|^{2}|+|v_{1}
-\bar{v}_{1}|^{2}|)^{p/2}\xi_{i}^{(p)}+\frac{p}{2}\mathbb{E}\int_{0}^{t}e^{\eta_{p}\epsilon}(\varepsilon
+|S_{\epsilon}^{s_{1},i}-S_{\epsilon}^{\bar{s}_{1},i}|^{2}\\
&+|I_{\epsilon}^{i_{1},i}-I_{\epsilon}^{\bar{i}_{1},i}|^{2}+|V_{\epsilon}^{v_{1},i}
-V_{\epsilon}^{\bar{v}_{1},i}|^{2})^{p/2-1}\xi^{(p)}_{\Lambda_{\epsilon}^{i}}\times\bigg\{[\epsilon \alpha(\Lambda_{\epsilon}^{i})+\beta(\Lambda_{\epsilon}^{i}) \epsilon +\frac{1}{\epsilon}-2(\mu(\Lambda_{\epsilon}^{i})+u_{1})] \\
&\times|S_{\epsilon}^{s_{1},i}
-S_{\epsilon}^{\bar{s}_{1},i}|^{2}+[\frac{1}{\epsilon}+\beta(\Lambda_{\epsilon}^{i}) \epsilon+(1-e)\beta(\Lambda_{\epsilon}^{i}) \epsilon+\frac{1}{\epsilon(1+\eta(\Lambda_{\epsilon}^{i}) I_{\epsilon}^{s_{1},i})}-2(\mu(\Lambda_{\epsilon}^{i})\\
&+\alpha(\Lambda_{\epsilon}^{i}))]|I_{\epsilon}^{i_{1},i}
-I_{\epsilon}^{\bar{i}_{1},i}|^{2}+[u_{1}\epsilon+(1-e)\beta(\Lambda_{\epsilon}^{i}) \epsilon +u_{1}\epsilon+m(\Lambda_{\epsilon}^{i}) u_{2}\epsilon\\
&-2\mu(\Lambda_{\epsilon}^{i})]|V_{\epsilon}^{v_{1},i}
-V_{\epsilon}^{\bar{v}_{1},i}|^{2}+\varepsilon_{2}|u_{\epsilon}^{v_{1},i}-u_{\epsilon}^{{\bar{v}_{1},i}}|^{2}\bigg\}d\epsilon\\
\leq &(\varepsilon+|s_{1}-\bar{s}_{1}|^{2}+|i_{1}-\bar{i}_{1}|^{2}|+|v_{1}-\bar{v}_{1}|^{2})^{p/2}\xi_{i}^{(p)}\\
&+\frac{p}{2}\mathbb{E}\int_{0}^{t}e^{\eta_{p}\epsilon}(\varepsilon+|S_{\epsilon}^{s_{1},i}
-S_{\epsilon}^{\bar{s}_{1},i}|^{2}+|I_{\epsilon}^{i_{1},i}-I_{\epsilon}^{\bar{i}_{1},i}|^{2}
+|V_{\epsilon}^{v_{1},i}-V_{\epsilon}^{\bar{v}_{1},i}|^{2})^{p/2}\xi^{(p)}_{\Lambda_{\epsilon}^{i}}\\
&\times\bigg\{\frac{C_{1}|S_{\epsilon}^{s_{1},i}-S_{\epsilon}^{\bar{s}_{1},i}|^{2}
+\tau_{1}|I_{\epsilon}^{i_{1},i}-I_{\epsilon}^{\bar{i}_{1},i}|^{2}+\tau_{2}|V_{\epsilon}^{v_{1},i}
-V_{\epsilon}^{\bar{v}_{1},i}|^{2}+\varepsilon_{2}|u_{\epsilon}^{v_{1},i}
-u_{\epsilon}^{{\bar{v}_{1},i}}|^{2}}
{\varepsilon+|S_{\epsilon}^{s_{1},i}-S_{\epsilon}^{\bar{s}_{1},i}|^{2}+|I_{\epsilon}^{i_{1},i}
-I_{\epsilon}^{\bar{i}_{1},i}|^{2}+|V_{\epsilon}^{v_{1},i}
-V_{\epsilon}^{\bar{v}_{1},i}|^{2}}\bigg\}d\epsilon,\\
\end{split}
\end{equation*}
where $C_{i}$, $i=1,2,3$ have been explained before. In addition, using  (\ref{bnmp}), we can get
\begin{equation}\label{mtpoj97}
\begin{split}
&e^{\eta_{p}t}\mathbb{E}((\varepsilon+|S_{t}^{s_{1},i}-S_{t}^{\bar{s}_{1},i}|^{2}+|I_{t}^{i_{1},i}
-I_{t}^{\bar{i}_{1},i}|^{2}+|V_{t}^{v_{1},i}-V(x,t)^{\bar{v}_{1},i}|^{2})^{p/2}\xi^{(p)}_{\Lambda_{t}^{i}})\\
\leq & (\varepsilon+|s_{1}-\bar{s}_{1}|^{2}+|i_{1}-\bar{i}_{1}|^{2}+|v_{1}-\bar{v}_{1}|^{2})^{p/2}\xi_{i}^{(p)}\\
&+\frac{p}{2}C\mathbb{E}\int_{0}^{t}e^{\eta_{p}\epsilon}(\varepsilon+|S_{\epsilon}^{s_{1},i}
-S_{\epsilon}^{\bar{s}_{1},i}|^{2}+|I_{\epsilon}^{i_{1},i}-I_{\epsilon}^{\bar{i}_{1},i}|^{2}
+|V_{\epsilon}^{v_{1},i}-V_{\epsilon}^{\bar{v}_{1},i}|^{2})^{p/2}\xi^{(p)}_{\Lambda_{\epsilon}^{i}}\\
&\times\bigg\{1-\varepsilon(\varepsilon+|S_{\epsilon}^{s_{1},i}-S_{\epsilon}^{\bar{s}_{1},i}|^{2}
+|I_{\epsilon}^{i_{1},i}-I_{\epsilon}^{\bar{i}_{1},i}|^{2}+|V_{\epsilon}^{v_{1},i}
-V_{\epsilon}^{\bar{v}_{1},i}|^{2})^{-1}\bigg\}d\epsilon\\
\leq &(\varepsilon+|s_{1}-\bar{s}_{1}|^{2}+|i_{1}-\bar{i}_{1}|^{2}|+|v_{1}
-\bar{v}_{1}|^{2})^{p/2}\xi_{i}^{(p)}+C\varepsilon^{p/2}e^{\eta_{p}t},
\end{split}
\end{equation}
when $\varepsilon \rightarrow 0$, we have the following result
\begin{equation}\label{87mop}
\begin{split}
&\mathbb{E}(|S_{t}^{s_{1},i}-S_{t}^{\bar{s}_{1},i}|^{p}+|I_{t}^{i_{1},i}-I_{t}^{\bar{i}_{1},i}|^{p}
+|V_{t}^{v_{1},i}-V_{t}^{\bar{v}_{1},i}|^{p})\\
\leq& C(|s_{1}-\bar{s}_{1}|^{p}+|i_{1}-\bar{i}_{1}|^{p}+|v_{1}-\bar{v}_{1}|^{p})e^{-\eta_{p}t}.
\end{split}
\end{equation}
Define the stopping time $$ \tau =\inf\{t\geq 0:\Lambda_{t}^{i}=\Lambda_{t}^{j}\}.$$
According to the definition of  $\mathbb{S}$ and irreducibility of $Q$, there exists $\theta>0$ such that
\begin{equation}\label{mo99}
\begin{split}
\mathbb{P}(\tau>t)\leq e^{-\theta t}, \quad \quad t>0.
\end{split}
\end{equation}

Due to $p\in (0,p_{0})$, and choose $ q>1$ such that $0< pq<p_{0}$, where $p_{0}$ is introduced in Eq.(\ref{78y5r}). It follows from H$\ddot{o}$lder's inequality that
\begin{equation}\label{6437olt}
\begin{split}
&\mathbb{E}(|S_{t}^{s_{1},i}-S_{t}^{\bar{s}_{1},j}|^{p}+|I_{t}^{i_{1},i}-I_{t}^{\bar{i}_{1},j}|^{p}
+|V_{t}^{v_{1},i}-V_{t}^{\bar{v}_{1},j}|^{p})\\
=&\mathbb{E}(|S_{t}^{s_{1},i}-S_{t}^{\bar{s}_{1},j}|^{p}{\bf 1}_{\{\tau>t/2\}})+\mathbb{E}(|S_{t}^{s_{1},i}-S_{t}^{\bar{s}_{1},j}|^{p}{\bf1}_{\{\tau\leq t/2\}})+\mathbb{E}(|I_{t}^{i_{1},i}-I_{t}^{\bar{i}_{1},j}|^{p}{\bf 1}_{\{\tau>t/2\}})\\
&+\mathbb{E}(|I_{t}^{i_{1},i}-I_{t}^{\bar{i}_{1},j}|^{p}{\bf 1}_{\{\tau\leq t/2\}})+\mathbb{E}(|V_{t}^{v_{1},i}-V_{t}^{\bar{v}_{1},j}|^{p}{\bf 1}_{\{\tau>t/2\}})+\mathbb{E}(|V_{t}^{v_{1},i}-V_{t}^{\bar{v}_{1},j}|^{p}{\bf1}_{\{\tau\leq t/2\}})\\
\leq&(\mathbb{E}|S_{t}^{s_{1},i}-S_{t}^{\bar{s}_{1},j}|^{pq}{\bf 1}_{\{\tau>t/2\}})^{1/q}(\mathbb{P}(\tau>t/2))^{1/p}+\mathbb{E}({\bf1}_{\{\tau\leq t/2\}}\mathbb{E}|S_{t-\tau}^{S_{\tau}^{s_{1},i},\Lambda_{\tau}^{i}}
-S_{t-\tau}^{S_{\tau}^{\bar{s}_{1},j},\Lambda_{\tau}^{j}}|^{p})\\
&+(\mathbb{E}|I_{t}^{i_{1},i}-I_{t}^{\bar{i}_{1},j}|^{pq}{\bf 1}_{\{\tau>t/2\}})^{1/q}(\mathbb{P}(\tau>t/2))^{1/p}+\mathbb{E}({\bf1}_{\{\tau\leq t/2\}}\mathbb{E}|I_{t-\tau}^{I_{\tau}^{i_{1},i},\Lambda_{\tau}^{i}}
-I_{t-\tau}^{I_{\tau}^{\bar{i}_{1},j},\Lambda_{\tau}^{j}}|^{p})\\
&+(\mathbb{E}|V_{t}^{v_{1},i}-V_{t}^{\bar{s}_{1},j}|^{pq}{\bf 1}_{\{\tau>t/2\}})^{1/q}(\mathbb{P}(\tau>t/2))^{1/p}+\mathbb{E}({\bf1}_{\{\tau\leq t/2\}}\mathbb{E}|V_{t-\tau}^{V_{s}^{v_{1},i},\Lambda_{\tau}^{i}}
-V_{t-\tau}^{V_{s}^{\bar{v}_{1},j},\Lambda_{\tau}^{j}}|^{p}).
\end{split}
\end{equation}
Applying the result of Eq.(\ref{mo99}), one can obtain
\begin{equation}\label{55f}
\begin{split}
&\mathbb{E}(|S_{t}^{s_{1},i}-S_{t}^{\bar{s}_{1},j}|^{p}+|I_{t}^{i_{1},i}-I_{t}^{\bar{i}_{1},j}|^{p}
+|V_{t}^{v_{1},i}-V_{t}^{\bar{v}_{1},j}|^{p})\\
\leq&e^{-\frac{q-1}{2q}\theta t}(\mathbb{E}|S_{t}^{s_{1},i}-S_{t}^{\bar{s}_{1},j}|^{pq})^{\frac{1}{q}}+C\mathbb{E}({\bf1}_{\{\tau\leq t/2\}}e^{-\eta_{p}(t-\tau)}\mathbb{E}|S_{\tau}^{s_{1},i}-S_{\tau}^{\bar{s}_{1},j}|^{p})\\
&+e^{-\frac{q-1}{2q}\theta t}(\mathbb{E}|I_{t}^{i_{1},i}-I_{t}^{\bar{i}_{1},j}|^{pq})^{\frac{1}{q}}+C\mathbb{E}({\bf1}_{\{\tau\leq t/2\}}e^{-\eta_{p}(t-\tau)}\mathbb{E}|I_{\tau}^{i_{1},i}-I_{\tau}^{\bar{i}_{1},j}|^{p})\\
&+e^{-\frac{q-1}{2q}\theta t}(\mathbb{E}|V_{t}^{v_{1},i}-V_{t}^{\bar{v}_{1},j}|^{pq})^{\frac{1}{q}}+C\mathbb{E}({\bf1}_{\{\tau\leq t/2\}}e^{-\eta_{p}(t-\tau)}\mathbb{E}|V_{s}^{v_{1},i}-V_{s}^{\bar{v}_{1},j}|^{p})\\
\leq& e^{-\frac{q-1}{2q}\theta t}(\mathbb{E}|S_{t}^{s_{1},i}-S_{t}^{\bar{s}_{1},j}|^{pq})^{\frac{1}{q}}
+Ce^{-\frac{\eta_{p}}{2}t}\mathbb{E}|S_{\tau}^{s_{1},i}-S_{\tau}^{\bar{s_{1}},j}|^{p}
+e^{-\frac{q-1}{2q}\theta t}(\mathbb{E}|I_{t}^{i_{1},i}-I_{t}^{\bar{i}_{1},j}|^{pq})^{\frac{1}{q}}\\
&+Ce^{-\frac{\eta_{p}}{2}t}\mathbb{E}|I_{\tau}^{i_{1},i}-I_{\tau}^{\bar{i_{1}},j}|^{p}
+e^{-\frac{q-1}{2q}\theta t}(\mathbb{E}|V_{t}^{v_{1},i}-V_{t}^{\bar{v}_{1},j}|^{pq})^{1/q}
+Ce^{-\frac{\eta_{p}}{2}t}\mathbb{E}|V_{s}^{v_{1},i}-V_{s}^{\bar{v_{1}},j}|^{p}\\
\leq& C(1+|s_{1}|^{p}+|\bar{s}_{1}|^{p}+|i_{1}|^{p}+|\bar{i}_{1}|^{p}+|v_{1}|^{p}+|\bar{v}_{1}|^{p})e^{-\sigma t},
\end{split}
\end{equation}
where $\sigma:=\frac{(q-1)\theta }{2q}\wedge \frac{\eta_{p}}{2}$, and in the last step, it follows from Eq.(\ref{fgg564}) and Eq.(\ref{87mop}) such that $$\sup\limits_{t\geq 0}\mathbb{E}(|S_{t}^{s_{1},i}|^{pq}+|I_{t}^{i_{1},i}|^{pq}+|V_{t}^{v_{1},i}|^{pq})\leq C,$$ and $$\sup\limits_{t\geq 0}\mathbb{E}(|S_{t}^{\bar{s}_{1},j}|^{pq}+|I_{t}^{\bar{i}_{1},j}|^{pq}+|V_{t}^{\bar{v}_{1},j}|^{pq})\leq C.$$ Thus, we also have assertion $$\lim\limits_{t\rightarrow \infty}\mathbb{E}(|S_{t}^{s_{1},i}-S_{t}^{\bar{s}_{1},j}|^{p}+|I_{t}^{i_{1},i}-I_{t}^{\bar{i}_{1},j}|^{p}+|V_{t}^{v_{1},i}-V_{t}^{\bar{v}_{1},j}|^{p})=0.$$  By virtue of  Eq.(\ref{mo99})
\begin{equation}\label{67olt}
\begin{split}
\mathbb{P}(\Lambda_{t}^{i}\neq\Lambda_{t}^{j})=\mathbb{P}(\tau>t)\leq e^{-\theta t} \quad t>0.
\end{split}
\end{equation}
Next, according to Eq.(\ref{55f}) and Eq.(\ref{67olt})  that
\begin{equation}\label{234FRTY}
\begin{split}
&W_{p}(\delta_{((s_{1},i_{1},i_{1}),i)}\mathbb{P}_{t},\delta_{((\bar{s}_{1},\bar{i}_{1},\bar{v}_{1}),j)}\mathbb{P}_{t})\\
&\leq\mathbb{E}(|S_{t}^{s_{1},i}-S_{t}^{\bar{s}_{1},j}|^{p}+|I_{t}^{i_{1},i}-I_{t}^{\bar{i}_{1},j}|^{p}+|V_{t}^{v_{1},i}-V_{t}^{\bar{v}_{1},j}|^{p})+\mathbb{P}(\Lambda_{t}^{i}\neq\Lambda_{t}^{j})\\
&\leq C(1+|s_{1}|^{p}+|\bar{s}_{1}|^{p}+|i_{1}|^{p}+|\bar{i}_{1}|^{p}+|v_{1}|^{p}+|\bar{v}_{1}|^{p})e^{-\sigma t}+e^{-\theta t}\\
&\leq Ce^{-\sigma^{*} t},
\end{split}
\end{equation}
where $\sigma^{*}:=\sigma\wedge\theta$. Assume $\pi$, $\nu\in \mathcal{P}(H_{3}\times\mathbb{S})$ are invariant measures of $((S_{t}^{s_{1},i}, I_{t}^{i_{1},i}, V_{t}^{v_{1},i}),\Lambda_{t}^{i})$, it follows from Eq.(\ref{234FRTY}) that
\begin{equation*}\label{kop345}
\begin{split}
&W_{p}(\pi, \nu)
=W_{p}(\pi \mathbb{P}_{t}, \nu \mathbb{P}_{t})\\
\leq&\sum_{i,j=1}^{N}\int_{H_{3}\times\mathbb{S}}\int_{H_{3}\times\mathbb{S}}\pi(d(s_{1},i_{1},v_{1})\times \{i\})\nu(d(\bar{s}_{1},\bar{i}_{1},\bar{v}_{1})\times\{j\})W_{p}(\delta_{((s_{1},i_{1},v_{1}),i)}P_{t},\delta_{((\bar{s}_{1},\bar{i}_{1},\bar{v}_{1}),j)}P_{t}).
\end{split}
\end{equation*}
When $t\rightarrow \infty $, we find $W_{p}(\pi, \nu) \rightarrow 0$. Hence, uniqueness of invariant measure follows immediately. The proof of \autoref{mmy6} has been completed.
\end{proof}
\section{Necessary and sufficient conditions for near-optimal controls}\label{Section4}
In this section, we will introduce two lemmas at first, then we will obtain the necessary condition for the near-optimal control of the model \eqref{1`q1}.
\subsection{Some priori estimates of the susceptible, infected and vaccinated}\label{sec-estimats2}
\begin{lemma}\label{1234}
We  assume $|\Gamma|=1$ and $u\in \mathcal{U}_{ad}$, the solution $(S(t),I(t),V(t),\Lambda_{t})$ of \eqref{1`q1} satisfies
\begin{equation}\label{s1og}
\mathbb{E}\int_{\Gamma}(S(x,t)+I(x,t)+V(x,t))dx=N,~~\forall t \geq 0,
\end{equation}
where $N:=\int_{\Gamma}(S_{0}(x)+I_{0}(x)+V_{0}(x))dx$.
\end{lemma}
\begin{proof}
First, we define the linear operator $J: L^{2}(\Gamma,\mathbb{R})\to \mathbb{R}$ as following
\begin{equation*}\label{}
\begin{split}
\forall u \in L^{2}(\Gamma,\mathbb{R}^{3}), Ju:= \int_{\Gamma}u(x)dx.
\end{split}
\end{equation*}
By the properties of $e^{tA_{i}}$, we have
\begin{equation*}\label{}
\begin{split}
\forall u \in L^{2}(\Gamma,\mathbb{R}^{3}), J(e^{tA_{i}}u-u)=0,~~ Ju=Je^{tA_{i}}u, \forall i=1,2,3.
\end{split}
\end{equation*}
Moreover, for any $u\in L^{2}(\Gamma,\mathbb{R}^{3}), u\geq 0$ almost everywhere in $\Gamma$, we define $\|u\|_{1}=Ju$.
On the other hand, we get
\begin{equation*}\label{}
\begin{split}
&S(t)+I(t)+V(t)=e^{tA_{1}}S_{0}+e^{tA_{2}}I_{0}+e^{tA_{3}}V_{0}+\int_{0}^{t}e^{(t-s)A_{1}}S(s)I(s)dB_{1}(s)\\
&+\int_{0}^{t}e^{(t-s)A_{2}}S(s)I(s)dB_{2}(s){+\int_{0}^{t}e^{(t-s)A_{2}}V(s)I(s)dB_{4}(s)+\int_{0}^{t}e^{(t-s)A_{3}}V(s)I(s)dB_{3}(s).}
\end{split}
\end{equation*}
Then, applying to the operator $J$ to both sides and using the properties of stochastic integrals
\begin{equation*}\label{}
\begin{split}
&\|S(t)+I(t)+V(t)\|_{1}=\|S_{0}+I_{0}+V_{0}\|_{1}+\int_{0}^{t}J(e^{(t-s)A_{1}}S(s)I(s))dB_{1}(s)\\
&+\int_{0}^{t}J(e^{(t-s)A_{2}}S(s)I(s))dB_{2}(s){+\int_{0}^{t}J(e^{(t-s)A_{2}}V(s)I(s))dB_{4}(s)+\int_{0}^{t}J(e^{(t-s)A_{3}}V(s)I(s))dB_{3}(s).}
\end{split}
\end{equation*}
where  $J(e^{(t-s)A_{1}}S(s))$ in the stochastic integral is understood as the process with value in $L(L^{2}(\Gamma,\mathbb{R}^{3}),\mathbb{R}^{3})$, the space of linear operator from $L^{2}(\Gamma,\mathbb{R}^{3})\to \mathbb{R}^{3} $, for any $ u \in L^{2}(\Gamma,\mathbb{R}^{3})$ that is defined by
\begin{equation*}\label{}
\begin{split}
J(e^{(t-s)A_{1}}S(s)I(s))u:\int_{\Gamma}(e^{(t-s)A_{1}}S(s)I(s))(x)dx,
\end{split}
\end{equation*}
by similar methods for $J(e^{(t-s)A_{2}}S(s)I(s)),~J(e^{(t-s)A_{2}}V(s)I(s)),~\mbox{and}~ J(e^{(t-s)A_{3}}V(s)I(s))$. By taking the expectation to both sides and using the properties of stochastic integral, we have
\begin{equation*}\label{}
\begin{split}
\mathbb{E}\|S(t)+I(t)+V(t)\|_{1}=\|S_{0}+I_{0}+V_{0}\|_{1}.
\end{split}
\end{equation*}The proof is completed.
\end{proof}

%\begin{lemma}
%For any $u(t)\in \mathcal{U}_{ad}$, we have
%\begin{equation*}
%E\limsup\limits_{t \to \infty}\int_{\Gamma}[S(x,t)+I(x,t)+V(x,t)]dx<\infty.
%\end{equation*}
%\end{lemma}
%We give the invariant set $\mathcal{D}$ of the SIV system \eqref{1`q1}
%\begin{equation}
%\mathcal{D} =\{(S(x,t),I(x,t),V(x,t))\in X:\int_{\Gamma}(S(x,t)+I(x,t)+V(x,t))dx \leq 1\}.
%\end{equation}
%where $X=\mathcal{L}_{(0,x)}^{2}\times \mathcal{L}_{(0,x)}^{2}\times \mathcal{L}_{(0,x)}^{2}$ which denotes the space function.

\begin{lemma}\label{12334}
Let $(D_{1},D_{2},D_{3}):=(D_{1}(x),D_{2}(x),D_{3}(x))$, for all $\theta \geq 0$ and $0 < \kappa < 1$ satisfying $\kappa\theta < 1$, and $(u(x,t),u'(x,t))
\in \mathcal{U}_{ad}$, along with the corresponding trajectories $(X(x,t),~X'(x,t))$, there exists a constant $C = C(\theta, \kappa)$ such that
\begin{equation}\label{sog}
\sum_{i=1}^3\mathbb{E}\sup_{0\leq t \leq T}\ \int_{0}^{T}\int_{\Gamma}|X_{i}(x,t)-X'_{i}(x,t)|^{2}dxdt\leq C\sum_{i=1}^2d(u_{i}(x,t), u_{i}'(x,t))^{\kappa\theta}.
\end{equation}
\end{lemma}
\begin{proof}
We assume $\theta \geq 1$ and $\forall r >0$, use the elementary inequality, we can get
\begin{equation}\label{soh}
\begin{split}
&\mathbb{E}\sup_{0\leq t \leq r}\int_{0}^{T}\int_{\Gamma} |S(x,t)-S'(x,t)|^{2}dxdt\\
=&2[\int_{0}^{t}\int_{\Gamma}-\mu(x,\Lambda_{t})\langle S(x,t)-S'(x,t),S(x,t)-S'(x,t)\rangle dxds -\int_{0}^{t}\int_{\Gamma}\beta(x,\Lambda_{t}) I(x,t) \langle S(x,t) \\
&-S'(x,t),S(x,t)-S'(x,t)\rangle ds+\int_{0}^{t}\langle D_{1}\Delta (S(x,t)-S'(x,t)), S(x,t)-S'(x,t) \rangle dxds]\\
&+\int_{0}^{t}\int_{\Gamma}|\sigma(x,\Lambda_{t}) I (x,t)(S(x,t)-S'(x,t))|^{2}dxds-\int_{0}^{t}\int_{\Gamma}\langle \sigma(x,\Lambda_{t}) I(x,t)(S(x,t)-S'(x,t)), S(x,t)\\
&-S'(x,t)\rangle dxdB(t)-\int_{0}^{t}\int_{\Gamma}\langle u_{1}(x,t)S(x,t)-u'_{1}(x,t)S'(x,t), S(x,t)-S'(x,t) \rangle dxds.
\end{split}
\end{equation}
If $0 < \kappa < 1$, then $\theta \geq 1$, it follows from Cauchy-Schwartz's inequality that
\begin{equation}\label{spo}
\begin{split}
\mu(x,\Lambda_{t})\langle S(x,t)-S'(x,t), S(x,t)-S'(x,t)\rangle \leq C|S(x,t)-S'(x,t)|^{2},
%&\mathbb{E}\int_{0}^{r} \chi_{u_{1}(x,t)\neq u_{1}'(x,t)}(x,t) dt\\
%\leq & C\left(\mathbb{E}\int_{0}^{r} dt\right)^{1-\kappa\theta}
%\times\left(\mathbb{E}\int_{0}^{r} \chi_{u_{1}(x,t)\neq u_{1}'(x,t)}(x,t) dt\right)^{\kappa\theta}\\
%\leq & Cd(u_{1}(x,t), u_{1}'(x,t))^{\kappa\theta}.
\end{split}
\end{equation}
it follows from metric definition that
\begin{equation}
\langle D_{1}\Delta (S(x,t)-S'(x,t)), S(x,t)-S' (x,t)   \rangle ds =D_{1}\int_{0}^{t}\int_{\Gamma}-\| \nabla (S(x,t)-S'(x,t))  \|^{2}dxds,
\end{equation}
similarly,
\begin{equation}
\begin{split}
&\langle u_{1}(x,t)S(x,t)-u'_{1}(x,t)S'(x,t), S(x,t)-S'(x,t) \rangle \\
=&\langle u_{1}(x,t)S(x,t)-u_{1}(x,t)S'(x,t), S(x,t)-S'(x,t) \rangle \\
&+ \langle  (u(x,t) S'(x,t)-u'(x,t) S'(x,t))|_{\chi_{u(x,t)\neq u'(x,t)}}, S(x,t)-S'(x,t)  \rangle\\
\leq& C(|S(x,t)-S'(x,t)|^{2}+(|u'_{1}(x,t)-u_{1}(x,t)|^{2}+|S(x,t)-S'(x,t)|^{2})|_{\chi_{u(x,t)\neq u'(x,t)}})\\
\leq & C(|S(x,t)-S'(x,t)|^{2}+d(u(x,t),u'(x,t))).
\end{split}
\end{equation}
Because $I$ is bounded, we have
\begin{equation}
\begin{split}
&\mathbb{E} \sup[\int_{0}^{t}\int_{\Gamma}\langle \sigma (x,\Lambda_{t})I(x,t)(S(x,t)-S'(x,t)),S(x,t)-S'(x,t)   \rangle dxdB(t)]\\
\leq &C\mathbb{E}[\sup (|S(x,t)-S'(x,t)|^{2})^{\frac{1}{2}}(\int_{0}^{T}\int_{\Gamma}\| \sigma(x,\Lambda_{t}) I (x,t)(S(x,t)-S'(x,t))  \|_{2}^{2}dxdt)^{\frac{1}{2}}]\\
\leq &C\mathbb{E}[\sup |S(x,t)-S'(x,t)|^{2}+\int_{0}^{T}\int_{\Gamma}\|S(x,t)-S' (x,t)  \|_{2}^{2}dxdt],
\end{split}
\end{equation}
where
\begin{equation}
\beta(x,\Lambda_{t}) I(x,t) \langle S(x,t)-S'(x,t), S(x,t)-S'(x,t)\rangle \leq C(\| S(x,t)-S'(x,t)  \|_{2}^{2}),
\end{equation}
and
\begin{equation}
\mathbb{E} \sup\limits_{0 \leq t \leq T}|S(x,t)-S'(x,t)|^{2}\leq Cd(u(x,t)-u'(x,t))^{\kappa},
\end{equation}
combined with the Cauchy-Schwartz inequality, we get
\begin{equation}
\begin{split}
\mathbb{E}\sup\limits_{0\leq t \leq T}|S(x,t)-S'(x,t)|^{\theta} &\leq (\mathbb{E}|^{\frac{2}{2-\theta}})^{\frac{2-\theta}{2}}(\mathbb{E}(\sup\limits_{0\leq t \leq T}|S(x,t)-S'(x,t)|^{\theta})^{\frac{2}{\theta}})^{\frac{\theta}{2}}\\
&\leq (\mathbb{E}\sup\limits_{0 \leq t \leq T }|S(x,t)-S'(x,t)|^{2})^{\frac{\theta}{2}}\leq Cd(u(x,t),u'(x,t))^{\frac{\kappa \theta}{2}}.
\end{split}
\end{equation}
Similarly, we can get
\begin{equation}\label{slh}
\begin{split}
&\mathbb{E}\sup_{0\leq t \leq r}\ |I(x,t)-I'(x,t)|^{\theta}\leq Cd(u(x,t),u'(x,t))^{\frac{\kappa \theta}{2}},\\
&\mathbb{E}\sup_{0\leq t \leq r}\ |V(x,t)-V'(x,t)|^{\theta}\leq Cd(u(x,t),u'(x,t))^{\frac{\kappa \theta}{2}}].
\end{split}
\end{equation}
%Then, summing up \eqref{soh} and \eqref{slh}, we may obtain that
%\begin{equation*}
%\small
%\sum_{i=1}^3 \mathbb{E}\sup_{0\leq t \leq r}\ |X_{i}(x,t)-X'_{i}(x,t)|^{2\theta}\leq C \bigg[\int_{0}^{r} \sum_{i=1}^3 \mathbb{E} \sup_{0\leq t \leq s} |X_{i}(x,t)-X'_{i}(x,t)|^{2\theta}ds +  \sum_{i=1}^2d(u_{i}(x,t), u_{i}'(x,t))^{\kappa\theta}\bigg].
%\end{equation*}
%Now by considering $0 \leq \theta < 1$, with the help of Cauchy-Schwartz's inequality, we have
%\begin{equation}\label{sll}
%\small
%\begin{split}
%\sum_{i=1}^3 \mathbb{E}\sup_{0\leq t \leq r}\ |X_{i}(x,t)-X'_{i}(x,t)|^{2\theta}&\leq  \sum_{i=1}^3 \left[\mathbb{E}\sup_{0\leq t \leq r}\ |X_{i}(x,t)-X'_{i}(x,t)|^{2}\right]^{\theta}\\
%& \leq C \bigg[\int_{0}^{r} \sum_{i=1}^3 \mathbb{E} \sup_{0\leq t \leq s} |X_{i}(x,t)-X'_{i}(x,t)|^{2\theta}ds+ \sum_{i=1}^2d(u_{i}(x,t), u_{i}'(x,t))^{\kappa}\bigg]^{\theta} \\&\leq  C^{\theta} \bigg[\sum_{i=1}^2d(u_{i}(x,t), u_{i}'(x,t))^{\kappa\theta}\bigg].
%\end{split}
%\end{equation}
%Therefore, the result is true by making use of Gronwall's inequality.\\
The proof is completed.
\end{proof}

Next, we will draw into the adjoint equation as follows:
%\begin{equation}\label{adj}
%\begin{cases}
%dp_{1}(x,t)=-b_{1}(X(x,t),u(x,t),p(x),q(x,t))dt + q_{1}(x,t)dB(t),\\
%dp_{2}(x,t)=-b_{2}(X(x,t),u(x,t),p(x),q(x,t))dt + q_{2}(x,t)dB(t),\\
%dp_{3}(x,t)=-b_{3}(X(x,t),u(x,t),p(x),q(x,t))dt+ q_{3}(x,t)dB(t),\\
%p_{i}(x,T)=h_{X_{i}}(X(x,T)),\ i=1,2,3,\\
%p_{i}(x,t)=0,x \in\partial \Gamma,
%\end{cases}
%\end{equation}
%where
\begin{equation}\label{adj}
\begin{split}
dp_{1}(x,t)=&-A_{1}+\bigg(-(\mu(x,\Lambda_{t})+u_1(x,t))S(x,t)+\beta(x,\Lambda_{t}) I(x,t) \bigg) p_{1}(x,t)+D_{1}\Delta p_{1}(x,t)\\
&+\beta(x,\Lambda_{t}) I(x,t) p_{2}(x,t)+u_1(x,t)  p_{3}(x,t)-\sigma(x,\Lambda_{t})  I(x,t) q_{1}(x,t)\\
&+\sigma(x,\Lambda_{t}) I(x,t) q_{2}(x,t)+ q_{1}(x,t)dB(t),\\
dp_{2}(x,t)=&-A_{2}+(\alpha(x,\Lambda_{t})-\beta(x,\Lambda_{t}) S(x,t))p_{1}(x,t)+\bigg(\beta(x,\Lambda_{t}) S(x,t)+  (1-e)\beta(x,\Lambda_{t}) V(x,t)
 \\&-(\mu(x,\Lambda_{t})+\alpha(x,\Lambda_{t}))-\frac{m(x,\Lambda_{t})u_{2}(x,t)}{(1+\eta(x,\Lambda_{t}) I(x,t))^{2}}\bigg)p_{2}(x,t)+D_{2}\Delta p_{2}(x,t)
 \\
\end{split}
\end{equation}
\begin{equation*}
\begin{split}
&-\bigg((1-e)\beta(x,\Lambda_{t})    V(x,t)-\frac{m(x,\Lambda_{t})u_{2}(x,t)}{(1+\eta(x,\Lambda_{t}) I(x,t))^{2}}\bigg) p_{3}(x,t)-\sigma(x,\Lambda_{t})  S(x,t) q_{1}(x,t)\\
 &+(\sigma(x,\Lambda_{t})  S(x,t)+(1-e)\sigma (x,\Lambda_{t}) V(x,t))q_{2}(x,t) -(1-e)\sigma (x,\Lambda_{t}) V(x,t) q_{3}(x,t)+ q_{2}(x,t)dB(t),\\
dp_{3}(x,t)=&(1-e)\beta(x,\Lambda_{t}) I(x,t) p_{2}(x,t)-\bigg(\mu (x,\Lambda_{t}) +(1-e)\beta  (x,\Lambda_{t})   I(x,t)\bigg)p_{3}(x,t)\\
&+D_{3}\Delta p_{3}(x,t)+(1-e)\sigma(x,\Lambda_{t})  I(x,t) q_{2}(x,t)-(1-e)\sigma(x,\Lambda_{t}) I(x,t) q_{3}(x,t)+ q_{3}(x,t)dB(t),\\
p_{i}(x,T)=&h_{X_{i}}(X(x,T)),\ i=1,2,3,\\
p_{i}(x,t)=&0,x \in\partial \Gamma.
\end{split}
\end{equation*}

\begin{lemma}\label{45w}
For all $0 < \kappa < 1$ and $0 < \theta< 1$ satisfying
$(1 + \kappa)\theta < 2$, and $u(x,t),u'(x,t)\in \mathcal{U}_{ad}$, along with the corresponding trajectories $(X(x,t),X'(x,t))$, and the solution $(p(x),q(x,t)),(p'(x,t),q'(x,t))$ of corresponding adjoint equation \eqref{adj}, there exists a constant
$C = C(\kappa, \theta) > 0$ such that
\begin{equation}\label{egw}
\begin{split}
&\sum_{i=1}^3 \mathbb{E}\int_{0}^{T}\int_{\Gamma}\ |p_{i}(x,t)-p'_{i}(x,t)|^{\theta}dxdt + \sum_{i=1}^3 \mathbb{E}\int_{0}^{T}\int_{\Gamma} |q_{i}(x,t)-q'_{i}(x,t)|^{\theta}dxdt\\
\leq& C\sum_{i=1}^2d(u_{i}(x,t), u_{i}'(x,t))^\frac{{\kappa\theta}}{2}.
\end{split}
\end{equation}
\end{lemma}
\begin{proof}
Let $\widehat{p}_{i}(x,t) \equiv p_{i}(x,t)-p'_{i}(x,t), ~\widehat{q}_{j}(x,t) \equiv q_{j}(x,t)-q'_{j}(x,t)(i,j=1,2,3)$. From the adjoint equation \eqref{adj}, we have
\begin{small}
\begin{equation*}\label{xgw}
\begin{cases}
\begin{split}
d \widehat{p}_{1}(x,t)=&-\bigg[A_{1}-\bigg(\mu(x,\Lambda_{t})+u_1(x,t)+ \beta(x,\Lambda_{t}) I(x,t)\bigg)\widehat{p}_{1}(x,t)+  D_{1}\Delta \widehat{p}_{1}(x,t)+   \beta(x,\Lambda_{t}) I(x,t)\widehat{p}_{2}(x,t)\\
&+u_1(x,t) \widehat{p}_{3}(x,t)-\sigma(x,\Lambda_{t}) I(x,t)\widehat{q}_{1}(x,t)+\sigma(x,\Lambda_{t}) I (x,t)  \widehat{q}_{2}(x,t)+\widehat{f}_{1}(x,t)\bigg]dt+ \widehat{q}_{1}(x,t)dB(t),\\
d \widehat{p}_{2}(x,t)=&-\bigg[A_{2}+(\alpha(x,\Lambda_{t}) -\beta(x,\Lambda_{t}) S(x,t))\widehat{p}_{1}(x,t)+\bigg(\beta(x,\Lambda_{t}) S(x,t)-(1-e)\beta(x,\Lambda_{t}) V(x,t)\\
&-(\mu(x,\Lambda_{t})+\alpha(x,\Lambda_{t}))-\frac{m(x,\Lambda_{t})u_{2}(x,t)I(x,t)}{1+\eta(x,\Lambda_{t}) I(x,t)}\bigg)\widehat{p}_{2}(x,t)+ D_{2}\Delta \widehat{p}_{2}(x,t)-\bigg((1-e)\beta(x,\Lambda_{t}) V(x,t)\\
&-\frac{m(x,\Lambda_{t})u_{2}(x,t)I(x,t)}{1+\eta(x,\Lambda_{t}) I(x,t)} \bigg)\widehat{p}_{3}(x,t)-\sigma (x,\Lambda_{t})S(x,t)\widehat{q}_{1}(x,t)+(\sigma(x,\Lambda_{t}) S(x,t)\\
&+(1-e)\sigma(x,\Lambda_{t})   V(x,t))\widehat{q}_{2}(x,t)-((1-e)\sigma (x,\Lambda_{t})V(x,t))\widehat{q}_{3}(x,t)+\widehat{f}_{2}(x,t)   \bigg]dt+ \widehat{q}_{2}(x,t)dB(t),\\
d \widehat{p}_{3}(x,t)=& -\bigg[((1-e)\beta (x,\Lambda_{t})I(x,t))\widehat{p}_{2}(x,t)-(\mu(x,\Lambda_{t})   +(1-e)\beta(x,\Lambda_{t})    I (x,t)  \widehat{p}_{3}(x,t)+ D_{3}\Delta \widehat{p}_{3}(x,t) \bigg]dt\\
&+((1-e)\sigma (x,\Lambda_{t})I(x,t))\widehat{q}_{2}(x,t)+((1-e)   \sigma(x,\Lambda_{t}) I(x,t))\widehat{q}_{3}(x,t)+    \widehat{f}_{3}(x,t)   +\widehat{q}_{3}(x,t)dB(t).\\
\end{split}
\end{cases}
\end{equation*}
\end{small}
where
\begin{footnotesize}
\begin{equation*}
\begin{split}
\widehat{f}_{1}(x,t)=&\beta(x,\Lambda_{t})     (I(x,t)-I'(x,t))(p'_{2}(x,t)-p'_{1}(x,t))+\sigma(x,\Lambda_{t})(I(x,t)-I'(x,t))(q'_{2}(x,t)
-q'_{1}(x,t))    ,\\
\widehat{f}_{2}(x,t)=&\beta(x,\Lambda_{t})     (S(x,t)-S'(x,t))(p'_{2}(x,t)-p'_{1}(x,t))+(1-e)\beta(x,\Lambda_{t})(V(x,t)-V'(x,t))(p'_{3}(x,t)
-p'_{2}(x,t))   \\&
+m(x,\Lambda_{t})x\bigg(\frac{u_{2}(x,t)}{(1+\eta(x,\Lambda_{t}) I(x,t))^{2}}-\frac{u_{2}'(x,t)}{(1+\eta(x,\Lambda_{t}) I'(x,t))^{2}}\bigg)(p'_{3}(x,t)-p'_{2}(x,t))+\sigma(x,\Lambda_{t})(S(x,t)\\&
-S'(x,t))(q'_{2}(x,t)-q'_{1}(x,t))+(1-e)\sigma(x,\Lambda_{t})(V(x,t)-V'(x,t))(q'_{3}(x,t)
-q'_{2}(x,t))    ,\\
\widehat{f}_{3}(x,t)=&(1-e)\beta(x,\Lambda_{t})(I(x,t)-I'(x,t))(p'_{2}(x,t)-p'_{3}(x,t)) +(1-e)\sigma(x,\Lambda_{t})(I(x,t)-I'(x,t))(q'_{2}(x,t)-q'_{3}(x,t)).
\end{split}
\end{equation*}
\end{footnotesize}
We assume the solution of the following stochastic differential equation is
$\phi(x,t):= (\phi_{1}(x,t),\phi_{2}(x,t),\phi_{3}(x,t))^{\top}$,
\begin{small}
\begin{equation*}
\begin{cases}
\begin{split}
d \phi_{1}(x,t)=&\bigg[(\mu(x,\Lambda_{t})+u_1(x,t) + \beta(x,\Lambda_{t}) I(x,t) )\phi_{1}(x,t)+D_{1}\Delta \phi_{1}(x,t)+(\alpha(x,\Lambda_{t})
\\&-\beta(x,\Lambda_{t}) S(x,t))\phi_{2}(x,t)+|\widehat{p}_{1}(x,t)|^{\theta-1}sgn(\widehat{p}_{1}(x,t))\bigg]dt-\bigg[(\sigma(x,\Lambda_{t}) I(x,t))\phi_{1}(x,t)
\\&+(\sigma(x,\Lambda_{t}) S(x,t))\phi_{2}(x,t)+|\widehat{q}_{1}(x,t)|^{\theta-1}sgn(\widehat{q}_{1}(x,t))\bigg]dB(t),\\
d \phi_{2}(x,t)=&\bigg[\beta(x,\Lambda_{t}) I(x,t)\phi_{1}(x,t)\bigg(\beta(x,\Lambda_{t}) S(x,t)-(1-e)\beta(x,\Lambda_{t}) V(x,t)-(\mu(x,\Lambda_{t})+\alpha(x,\Lambda_{t}))
\\&-\frac{m(x,\Lambda_{t})u_{2}(x,t)I(x,t)}{1+\eta(x,\Lambda_{t}) I(x,t)}\bigg)\phi_{2}(x,t)+D_{2}\Delta \phi_{2}
+(1-e)\beta (x,\Lambda_{t})   I(x,t) \phi_{3}(x,t)
\\&+|\widehat{p}_{2}(x,t)|^{\theta(x,t)-1}sgn(\widehat{p}_{2}(x,t))\bigg]dt+\bigg[\sigma(x,\Lambda_{t}) I(x,t) \phi_{3}(x,t)+(\sigma(x,\Lambda_{t}) S(x,t)\\
&+(1-e)\sigma(x,\Lambda_{t}) V(x,t))\phi_{2}(x,t)
%\end{split}
%\end{cases}
%\end{equation}
%\begin{equation}
%\begin{cases}
%\begin{split}
+(1-e)\sigma(x,\Lambda_{t}) I(x,t)\phi_{3}(x,t)+|\widehat{q_{2}}|^{\theta-1}sgn(\widehat{q}_{2}(x,t)\bigg] dB(t),\\
d \phi_{3}(x,t)= & \bigg[u_1 \phi_{1}(x,t)-\bigg((1-e)\beta(x,\Lambda_{t}) V(x,t)-\frac{m(x,\Lambda_{t})u_{2}(x,t)I(x,t)}{1+\eta(x,\Lambda_{t}) I(x,t)}\bigg)\phi_{2}(x,t)
-(\mu(x,\Lambda_{t})+(1-\\
&e)\beta (x,\Lambda_{t})I(x,t))\phi_{3}(x,t)+D_{3}\Delta \phi_{3}(x,t)+|\widehat{p}_{3}(x,t)|^{\theta-1}sgn(\widehat{p}_{3}(x,t))\bigg]dt-\bigg(((1-e)\sigma (x,\Lambda_{t})\times\\
&V(x,t))\phi_{1}(x,t)+((1-e)\sigma (x,\Lambda_{t})I(x,t))    \phi_{2}(x,t)+((1-e)\sigma (x,\Lambda_{t})I(x,t))    \phi_{2}(x,t)\\
&+|\widehat{q}_{3}(x,t)|^{\theta-1}sgn(\widehat{q}_{3}(x,t))\bigg)dB(t),
\end{split}
\end{cases}
\end{equation*}
\end{small}
where $sgn(\cdot)$ is a symbolic function.

%\subsection{Some priori estimates of the susceptible, infected and recovered}\label{sec-estimats1}
\begin{lemma}\label{1234}
For any $\theta \geq 0$ and $u(x,t)\in \mathcal{U}_{ad}$, we have
\begin{equation}\label{s1og}
\mathbb{E}\sup_{0\leq t \leq T}\ |S(x,t)|^{\theta}\leq C, \ \mathbb{E}\sup_{0\leq t \leq T}\ |I(x,t)|^{\theta}\leq C, \ \mathbb{E}\sup_{0\leq t \leq T}\ |V(x,t)|^{\theta}\leq C,
\end{equation}
where $C$ is a constant that depends only on $\theta$.
\end{lemma}
\begin{proof}
Using It\^o's formula for $|S(x,t)|^{\theta}+|I(x,t)|^{\theta}+|V(x,t)|^{\theta}$, one has
%\begin{equation}
%\begin{split}
%|S(x,t)|^{\alpha}&+|I(x,t)|^{\alpha}+|V(x,t)|^{\alpha}=|S(x,0)|^{\alpha}+|I(x,0)|^{\alpha}+|V(x,0)|^{\alpha}\\
%&+\int_{0}^{t}\bigg(\theta|S|^{\theta -1}(1-p)\mu+\theta|S|^{\theta -1}\alpha I-\theta|S|^{\theta-1}(\mu+u_{1})S-\theta|S|^{\theta-1}\beta SI\\
%&+\theta|S|^{\theta-1}D_{1}\Delta S+\frac{1}{2}\theta (\theta-1)|S|^{\theta-2}\sigma^{2}S^{2}I^{2}+\theta|I|^{\theta-1}\beta SI\\
%&+\theta |I|^{\theta -1}(1-e)\beta VI-\theta|I|^{\theta-1}(\mu+\alpha)I+\theta|I|^{\theta-1}D_{2}\Delta I -\theta|I|^{\theta-1}\frac{mu_{2}I}{1+\eta I}\\
%&+\frac{1}{2}\theta (\theta-1)|I|^{\theta-2}\sigma^{2}S^{2}I^{2} +\frac{1}{2}\theta (\theta-1)|I|^{\theta-2}(1-e)^{2}V^{2}I^{2}\\
%&+\theta|V|^{\theta-1}p\mu+\theta|V|^{\theta-1}u_{1}S-\theta|V|^{\theta-1}\mu V-\theta|V|^{\theta-1}(1-e)\beta V I+\theta|V|^{\theta-1}D_{3}\Delta V\\
%&+\theta|V|^{\theta-1}\frac{mu_{2}I}{1+\eta I}+\frac{1}{2}\theta (\theta-1)|V|^{\theta}(1-e)^{2}V^{2}I^{2} \bigg)dt\\
%&-(\theta|S|^{\theta-1}\sigma S I-\theta|I|^{\theta-1}(\sigma S I+(1-e)\sigma V I)+\theta|V|^{\theta-1}(1-e)\sigma V I)dB(t),
%\end{split}
%\end{equation}
\begin{footnotesize}
\begin{equation}\label{[po}
\begin{split}
&|S(x,t)|^{\theta}+|I(x,t)|^{\theta}+|V(x,t)|^{\theta}\\
\leq&|S(x,0)|^{\theta}+|I(x,0)|^{\theta}+|V(x,0)|^{\theta}+\int_{0}^{t}\int_{\Gamma}\bigg(\theta|S(x,t)|^{\theta -1}(1-p(x,\Lambda_{t}))b(x,\Lambda_{t})+\theta|S(x,t)|^{\theta -1}\alpha(x,\Lambda_{t}) I(x,t)\\
&-\theta|S(x,t)|^{\theta}(\mu(x,\Lambda_{t})+u_{1}(x,t))-\theta|S(x,t)|^{\theta}\beta(x,\Lambda_{t}) I(x,t)+\theta|S(x,t)|^{\theta-2}D_{1}\|\nabla S(x,t)\|^{2}\\
&+\frac{1}{2}\theta (\theta-1)|S(x,t)|^{\theta}\sigma^{2}(x,\Lambda_{t})I^{2}(x,t)+\theta|I(x,t)|^{\theta}\beta(x,\Lambda_{t}) S(x,t)+\theta |I(x,t)|^{\theta}(1-e)\beta(x,\Lambda_{t}) V(x,t)\\
&-\theta|I(x,t)|^{\theta}(\mu(x,\Lambda_{t})+\alpha(x,\Lambda_{t}))+\theta|I(x,t)|^{\theta-2}D_{2}\|\Delta I(x,t)\|^{2} -\theta|I(x,t)|^{\theta}\frac{m(x,\Lambda_{t})u_{2}(x,t)}{1+\eta(x,\Lambda_{t}) I(x,t)}\\
&+\frac{1}{2}\theta (\theta-1)|I(x,t)|^{\theta}\sigma^{2}(x,\Lambda_{t})S^{2}(x,t)+\frac{1}{2}\theta (\theta-1)|I(x,t)|^{\theta}(1-e)^{2}V^{2}(x,t)+\theta|V(x,t)|^{\theta-1}p(x,\Lambda_{t})b(x,\Lambda_{t})\\
& +\theta|V(x,t)|^{\theta-1}u_{1}(x,t)S(x,t)-\theta|V(x,t)|^{\theta}\mu(x,\Lambda_{t})-\theta|V(x,t)|^{\theta}(1-e)\beta(x,\Lambda_{t}) I(x,t)+\theta|V(x,t)|^{\theta-2}D_{3}\|\nabla V(x,t)\|^{2}\\
&+\theta|V(x,t)|^{\theta-1}\frac{m(x,\Lambda_{t})u_{2}(x,t)I(x,t)}{1+\eta(x,\Lambda_{t}) I(x,t)}+\frac{1}{2}\theta (\theta-1)|V(x,t)|^{\theta}(1-e)^{2}I^{2}(x,t) \bigg)dxdt-(\theta|S(x,t)|^{\theta}\sigma (x,\Lambda_{t})I(x,t)\\
&-\theta|I(x,t)|^{\theta}(\sigma(x,\Lambda_{t}) S(x,t)+(1-e)\sigma(x,\Lambda_{t}) V(x,t) I(x,t))+\theta|V(x,t)|^{\theta}(1-e)\sigma(x,\Lambda_{t}) I(x,t))dB(t)\\
\leq& |S(x,0)|^{\theta}+|I(x,0)|^{\theta}+|V(x,0)|^{\theta}+C\mathbb{E}\sup \int_{0}^{t}\int_{\Gamma}(|S(x,s)|^{\theta}+|I(x,s)|^{\theta}+|V(x,s)|^{\theta})dxds\\
&-\mathbb{E}\sup\int_{0}^{t}\int_{\Gamma}(\theta|S(x,t)|^{\theta}\sigma (x,\Lambda_{t})I(x,t)-\theta|I(x,t)|^{\theta}(\sigma(x,\Lambda_{t}) S(x,t)\\
&+(1-e)\sigma(x,\Lambda_{t}) V(x,t) I(x,t))+\theta|V(x,t)|^{\theta}(1-e)\sigma(x,\Lambda_{t}) I(x,t))dxdB(x,t).
\end{split}
\end{equation}
\end{footnotesize}
We take the upper bound and the expectation of \eqref{[po}, by the Burkholder-Davis-Gundy inequality, existence a constant $C$ such that
\begin{equation}
\begin{split}
&\mathbb{E}\sup\int_{0}^{t}\int_{\Gamma}(\theta|S(x,t)|^{\theta}\sigma(x,\Lambda_{t}) I(x,t)-\theta|I(x,t)|^{\theta}(\sigma(x,\Lambda_{t}) S(x,t)\\
&+(1-e)\sigma(x,\Lambda_{t}) V(x,t) I(x,t))+\theta|V(x,t)|^{\theta}(1-e)\sigma(x,\Lambda_{t}) I(x,t))dxdB(x,t)\\
=&\theta \sigma(x,\Lambda_{t})(|S(x,t)|^{\theta}I(x,t)+|I(x,t)|^{\theta}S(x,t))+\theta(1-e)\sigma(x,\Lambda_{t})(|I(x,t)|^{\theta}V(x,t)-|V(x,t)|^{\theta}I(x,t))\\
\leq & C[|S(x,t)|^{\theta}+|I(x,t)|^{\theta}+|V(x,t)|^{\theta}],
\end{split}
\end{equation}
hence, we have
\begin{equation}
\begin{split}
\mathbb{E}\sup\limits_{0\leq t \leq T}(|S(x,t)|^{\theta}&+|I(x,t)|^{\theta}+|V(x,t)|^{\theta})\leq 2\mathbb{E}(|S(x,0)|^{\theta}+|I(x,0)|^{\theta}+|V(x,0)|^{\theta})\\
&+C\mathbb{E}\sup\int_{0}^{t}\int_{\Gamma}(|S(x,s)|^{\theta}+|I(x,s)|^{\theta}+|V(x,s)|^{\theta})dxds.
\end{split}
\end{equation}
An application of the Cauchy-Schwartz inequality yields that
\begin{equation}
\begin{split}
&\mathbb{E}\sup\limits_{0\leq t \leq T}(|S(x,t)|^{\theta}+|I(x,t)|^{\theta}+|V(x,t)|^{\theta})\\
\leq &(\mathbb{E}1^{\frac{2}{2-\theta}})^{1-\frac{\theta}{2}}(\mathbb{E}(\sup(|S(x,t)|^{\theta}+|I(x,t)|^{\theta}+|V(x,t)|^{\theta}))^{\frac{2}{\theta}})^{\frac{\theta}{2}}ds\\
\leq & (\mathbb{E}\sup\limits_{0\leq t \leq T}(|S(x,t)|^{2}+|I(x,t)|^{2}+|V(x,t)|^{2})^{\frac{\theta}{2}})\leq C.
\end{split}
\end{equation}
Similarly, one has
\begin{equation*}
\mathbb{E}\sup_{0\leq t \leq T}\ |I(x,t)|^{\theta}\leq C,
\quad
\mathbb{E}\sup_{0\leq t \leq T}\ |V(x,t)|^{\theta}\leq C,
\end{equation*}
The proof is complete.
\end{proof}

\begin{lemma}\label{sib}
For any $u(x,t)$, $u'(x,t)\in \mathcal{U}_{ad}$, we have
\begin{equation}\label{stw}
\sum_{i=1}^3 \mathbb{E}\sup_{0\leq t \leq T}\ \int_{0}^{T}\int_{\Gamma}|p_{i}(x,t)|^{2}dxdt+\sum_{i=1}^3 \mathbb{E}\ \int_{0}^{T}\int_{\Gamma}\ |q_{i}(x,t)|^{2}dxdt \leq C,
\end{equation}
where $C$ is a constant.
\end{lemma}
We define a metric on the admissible control domain $\mathcal{U}_{ad}[0, T]$ as follows:
\begin{equation}\label{syw}
d(u(x,t), u'(x,t)) = \mathbb{E}[mes\{x\times t \in \Gamma\times [0; T] : u(x,t)\neq u'(x,t)\}] \ \forall u(x,t), u'(x,t) \in \mathcal{U}_{ad},
\end{equation}
where $mes$ represents Lebesgue measure. Since $U$ is closed, it can be shown  that $\mathcal{U}_{ad}$ is a complete metric space under $d$.
%\begin{lemma}\label{s2g}
%(Ekeland's principle) \cite{wsd}. Let $(Q, d)$ be a complete metric space and $F(\cdot): Q \rightarrow \mathbb{R}$ be a lower-semicontinuous and bounded from below. For any $\varepsilon > 0$, we assume that $u^{\varepsilon}(\cdot)\in Q$ satisfies
%\begin{equation*}
%F(u^{\varepsilon}(\cdot))\leq \inf_{u(\cdot)\in Q}F(u(\cdot))+\varepsilon.
%\end{equation*}
%Then there is $u^{\lambda}(\cdot)\in Q$ such that for all $\lambda > 0$ and $u(\cdot) \in Q$,
%\begin{equation*}
%F(u^{\lambda}(\cdot)) \leq F(u^{\varepsilon}(\cdot)),\  d(u^{\lambda}(\cdot), u^{\varepsilon}(\cdot)) \leq \lambda,\  and\ F(u^{\lambda}(\cdot)) \leq F(u(\cdot))+ \frac{\varepsilon}{\lambda}d(u^{\lambda}(\cdot), u^{\varepsilon}(\cdot)).
%\end{equation*}
%\end{lemma}

According to the definition of $L$ and  \autoref{sib}, the above equation admits a unique solution. Combining  with the above two estimates, the desired result then holds immediately.
\end{proof}
\subsection{Necessary conditions for near-optimal controls}\label{sec-controols}
\begin{theorem}\label{thm-stationary1}
Let $(p^{\varepsilon}(x,t), q^{\varepsilon}(x,t))$ is the solution of the adjoint equation \eqref{adj} under the control $u^{\varepsilon}(x,t)$. Then, there exists a constant $C$ such that for any
$\theta\in[0, 1),~\varepsilon > 0$ and any $\varepsilon$-optimal pair $(X^{\varepsilon}(x,t), u^{\varepsilon}(x,t))$, it holds that
\begin{footnotesize}
\begin{equation*}\label{fmm9w}
\begin{split}
\min_{u(x,t)\in\mathcal{U}_{ad}}\mathbb{E}\int_{0}^{T}\int_{\Gamma}\bigg( [u_{1}(x,t)-u_{1}^{\varepsilon}(x,t)]S^{\varepsilon}(x,t)(p_{3}^{\varepsilon}(x,t)-p_{1}^{\varepsilon}(x,t))+\frac{1}{2}(\tau_{1} [u_{1}(x,t)-u_{1}^{\varepsilon}(x,t)]^{2})\bigg)dxdt
\geq -C\varepsilon^{\frac{\theta}{3}}.
\end{split}
\end{equation*}
\end{footnotesize}
and
\begin{footnotesize}
\begin{equation*}\label{fmm9w}
\begin{split}
\min_{u(x,t)\in\mathcal{U}_{ad}}\mathbb{E}\int_{0}^{T}\int_{\Gamma}\bigg(& \frac{m(x,\Lambda_{t})[u_{2}(x,t)-u_{2}^{\varepsilon}(x,t)]I^{\varepsilon}(x,t)}{1+\eta(x,\Lambda_{t})
I^{\varepsilon}(x,t)}(p_{3}^{\varepsilon}(x,t)-p_{2}^{\varepsilon}(x,t))-\frac{1}{2}(\tau_{2} [u_{2}(x,t)-u_{2}^{\varepsilon}]^{2}(x,t))\bigg)dxdt
\geq -C\varepsilon^{\frac{\theta}{3}}.
\end{split}
\end{equation*}
\end{footnotesize}
\end{theorem}
\begin{proof}
The crucial step of the proof is to indicate that $H_{u(x,t)}(x,t,X^{\varepsilon},u^{\varepsilon}(x,t),p^{\varepsilon}(x,t),q^{\varepsilon}(x,t))$ is very small and use $\varepsilon$ to estimate it. The proof process is similar to \cite{3}, so we omit the details. This completes the proof.
\end{proof}

%\subsection{Sufficient conditions for near-optimal controls}\label{Section4}
\subsection{Sufficient conditions for near-optimal controls}\label{sec-controolsO}
We define the Hamiltonian function  $H(x,t,X(x,t),u(x,t),p(x),q(x,t)): \Gamma \times [0,T]\times\mathbb{R}^{3}_{+}\times\mathcal{U}_{ad}\times\mathbb{R}^{3}_{+}\times\mathbb{R}^{3}_{+}\rightarrow \mathbb{R}$ as follows:
\begin{equation}\label{xuow}
\begin{split}
H(x,t,X(x,t),u(x,t),p(x),q(x,t))&=f^{\top}(X(x,t),u(x,t))p(x)+\sigma_{*}^{\top}(X(x,t))q(x,t)+L(X(x,t),u(x,t)),
\end{split}
\end{equation}
with
\begin{equation*}
f(X(x,t),u(x,t))=
\begin{pmatrix}
	f_{1}(X(x,t),u(x,t)) \\
	f_{2}(X(x,t),u(x,t)) \\
	f_{3}(X(x,t),u(x,t)) \\
	\end{pmatrix},
 \    \sigma_{*}(X(x,t))=
\begin{pmatrix}
	\sigma_{14}(X(x,t)) \\
	\sigma_{24}(X(x,t)) \\
    \sigma_{34}(X(x,t)) \\
	\end{pmatrix},
\end{equation*}
where $f_{i} (i=1,2,3)$ and $\sigma_{i4} (i=1,2,3)$ are defined in \eqref{1`q1}, and $L(X(x,t),u(x,t))$ is defined in \eqref{s1}.

\begin{theorem}\label{thm-stationary}
Let $(X^{\varepsilon}(x,t), u^{\varepsilon}(x,t))$ be an admissible pair and $(p^{\varepsilon}(x,t),~q^{\varepsilon}(x,t))$ be the solutions of adjoint equation \eqref{adj} corresponding to $(X^{\varepsilon}(x,t), u^{\varepsilon}(x,t))$.
Assume $H(x,t, X(x,t), u(x,t), p(x), q(x,t)) $ is convex, a.s. For any $\varepsilon > 0$,  it follows that
\begin{footnotesize}
\begin{equation*}\label{fpppow}
\begin{split}
\varepsilon
\geq  \sup_{u^{\varepsilon}(x,t)\in\mathcal{U}_{ad}[0,T]}\mathbb{E}\int_{0}^{T}\int_{\Gamma} \bigg( [u_{1}^{\varepsilon}(x,t)-u_{1}(x,t)]S^{\varepsilon}(x,t)(p_{3}^{\varepsilon}(x,t)-p_{1}^{\varepsilon}(x,t)) +\frac{1}{2}(\tau_{1} [u_{1}^{\varepsilon}(x,t)-u_{1}(x,t)]^{2})\bigg)dxdt,
\end{split}
\end{equation*}
\end{footnotesize}
and
\begin{footnotesize}
\begin{equation*}\label{fpppow}
\begin{split}
\varepsilon
\geq  \sup_{u^{\varepsilon}(x,t)\in\mathcal{U}_{ad}[0,T]}\mathbb{E}\int_{0}^{T}\int_{\Gamma} \bigg(\frac{m(x,\Lambda_{t})[u_{2}^{\varepsilon}(x,t)-u_{2}(x,t)]I^{\varepsilon}(x,t)}{1+\eta(x,\Lambda_{t}) I^{\varepsilon}(x,t)}(p_{3}^{\varepsilon}(x,t)-p_{2}^{\varepsilon}(x,t))-\frac{1}{2}(\tau_{2} [u_{2}^{\varepsilon}(x,t)-u_{2}(x,t)]^{2})\bigg)dxdt,
\end{split}
\end{equation*}
\end{footnotesize}
then
\begin{equation*}
J(0,x_{0};u^{\varepsilon}(x,t))\leq \inf_{u(x,t)\in\mathcal{U}_{ad}[0,T]}J(0,x_{0};u(x,t))+C\varepsilon^{\frac{1}{2}}.
\end{equation*}
\end{theorem}
\begin{proof}
From the definition of $\widetilde{d}$ of \autoref{sib}, according to Holder's inequality and reference \cite{kiyt}, we can obtain the conclusion.
\end{proof}
%In the above sections, we have give the sufficient and necessary  conditions of the near-optimality control of SIV epidemic model. Our next goal is to illustrate the theoretical results through the numerical solution. In the following, we will give some figures to show the results.
\section{Optimal control and Numerical Algorithm}\label{Section5}
In this section, we study the optimal control of the stochastic SIV epidemic system, and then give the numerical approximation algorithm for the finite horizon optimal control problem of the stochastic epidemic system. %Online off-policy integral reinforcement learning (IRL) algorithm
\subsection{Optimal control}
Since $u_{1},u_{2}$ are function of the state and time, the easiest approach to the calculation of $(u_{1},u_{2},V^{*})$ is via the Bellman equation of continuous-time dynamic programming. Let $(p_{1},p_{2},p_{3}):=(V_{1{\bf{x}}},V_{2{\bf{x}}},V_{3{\bf{x}}})$ and $(P_{1},P_{2},P_{3}):=(V_{1{\bf{x}}{\bf{x}}},V_{2{\bf{x}}{\bf{x}}},V_{3{\bf{x}}{\bf{x}}})$.  Thus, the following Lyapunov equation (LE) is obtained.
\begin{equation}]\label{lkxcc}
\begin{split}
0&=V_{t}+p_{1}[D_{1}(x)\Delta S(x,t)+(1-p(x,\Lambda_{t}))b(x,\Lambda_{t})+\alpha(x,\Lambda_{t})  I(x,t)-\mu(x,\Lambda_{t}) S(x,t)\\
&-\beta(x,\Lambda_{t}) S(x,t) I(x,t)]+p_{2}[D_{2}(x)\Delta I(x,t)+\beta(x,\Lambda_{t}) S(x,t) I(x,t)\\
&  +(1-e)\beta(x,\Lambda_{t}) I(x,t) V(x,t)-(\mu(x,\Lambda_{t})+\alpha(x,\Lambda_{t}) )I(x,t)]\\
&+p_{3}[D_{3}(x)\Delta V(x,t)+ p(x,\Lambda_{t}) b(x,\Lambda_{t})-\mu(x,\Lambda_{t}) V(x,t)\\
&-(1-e)\beta(x,\Lambda_{t}) I(x,t) V(x,t)]+\frac{1}{2}[P_{1}\sigma^{2}(x,\Lambda_{t})S^{2}(x,t)+P_{2}\sigma^{2}(x,\Lambda_{t})I^{2}(x,t)\\
&+P_{3}\sigma^{2}(x,\Lambda_{t})V^{2}(x,t)]+A_{1}S(x,t)+A_{2}I(x,t)+S^{*}(x,t)+I^{*}(x,t),
\end{split}
\end{equation}
where
\begin{equation}
S^{*}(x,t)=p_{3}(x,t)u_{1}(x,t)S(x,t)-p_{1}(x,t)u_{1}(x,t)S(x,t)+\tau_{1}u_{1}^{2}(x,t),
\end{equation}
and
\begin{equation}
I^{*}(x,t)=p_{3}(x,t)\frac{m(x,\Lambda_{t})u_{2}(x,t)I(x,t)}{1+\eta(x,\Lambda_{t}) I(x,t)}-p_{2}(x,t)\frac{m(x,\Lambda_{t})u_{2}(x,t)I(x,t)}{1+\eta(x,\Lambda_{t}) I(x,t)}+\tau_{2}u_{2}^{2}(x,t),
\end{equation}
the regular control given as
\begin{equation}
u_{1}(x,t)=\frac{p_{1}(x,t)-p_{3}(x,t)}{2\tau_{1}}S(x,t),~~~u_{2}(x,t)=\frac{(p_{2}(x,t)-p_{3}(x,t))m(x,\Lambda_{t})}{2\tau_{2}(1+\eta (x,\Lambda_{t})I(x,t))}I(x,t).
\end{equation}
the optimal given as
\begin{equation}
\begin{split}
u_{1}^{*}(x,t)=
\begin{cases}
u_{1}^{\max},&u_{1}^{\max}\leq u_{1}(x,t),\\
u_{1}(x,t),&u_{1}^{\min}\leq u_{1}(x,t)\leq u_{1}^{\max}, \\
u_{1}^{\min},&u_{1}(x,t)\leq u_{1}^{\min},
\end{cases}
u_{2}^{*}(x,t)=
\begin{cases}
u_{2}^{\max},&u_{2}^{\max}\leq u_{2}(x,t),\\
u_{2}(x,t),&u_{2}^{\min}\leq u_{2}(x,t)\leq u_{2}^{\max}, \\
u_{2}^{\min},&u_{2}(x,t)\leq u_{2}^{\min}.
\end{cases}
\end{split}
\end{equation}
Furthermore, according to the ideas in Lenhart and Workman, the optimal control $u^{*}(x,t)$, which minimizes the objective function $J(0, X(x,t); u(x,t))$ is present as
$$u_{1}^{*}(x,t)=\max\{0,\min\{1, \frac{S(x,t)}{2\tau_{1}}(p_{1}(x,t)-p_{3}(x,t))\}\},$$
and
$$u_{2}^{*}(x,t)=\max\{0,\min\{1, \frac{m(x,\Lambda_{t})I(x,t)}{2\tau_{2}(1+\eta(x,\Lambda_{t}) I(x,t))}(p_{2}(x,t)-p_{3}(x,t))\}\},$$
where $(p_{1}(x,t),p_{2}(x,t),p_{3}(x,t))$ is the solution of system \eqref{adj}.
\subsection{Approximation of HJB equation}
It is well known that it is difficult to find the analytical solution to the HJB equation due to the partial differential equation. Therefore, it is necessary to find the successive approximation method for the finite horizon optimal control problem of the stochastic epidemic system. In this section, we assume that the spaces are homogeneous.  For simplify, we set $(S,I,V):=(S(x,t),I(x,t),V(x,t))$, respectively.  In order to give the online learning algorithm, we first give the Lemma.
\begin{lemma}\label{lemmav}
Assume $(u_{1}^{(i+1)},~u_{2}^{(i+1)})$ are the admissible control of the epidemic system, $V^{(i)}({\bf{x}},t)$ is the solution to the Lyapunov equation $LE(V^{(i)},u^{(i)})=0$ with the boundary conditions $V^{(i)}(0,t)=0,~V^{(i)}({\bf{x}}(t_{f}),t)=\phi({\bf{x}}(t_{f}))$, then \\
(1) $u_{1}^{(i+1)}=-(2\tau_{1})^{-1}S^{\top}(x,t)(V_{1{\bf{x}}}^{(i)}-V_{3{\bf{x}}}^{(i)}),~\mbox{and}~u_{2}^{(i+1)}
=-[2\tau_{2}(1+\eta(x,\Lambda_{t}) I(x,t))]^{-1}I^{\top}(x,t)m(x,\Lambda_{t})(V_{2{\bf{x}}}^{(i)}-V_{3{\bf{x}}}^{(i)})$ is an admissible control.\\ Moreover, if $V^{(i+1)}({\bf{x}},t)$ is the solution to the Lyapunov equation $LE(V^{(i+1)},u^{(i+1)})=0$ with the boundary conditions $V^{(i+1)}(0,t)=0,~V^{(i+1)}({\bf{x}}(t_{f}),t)=\phi({\bf{x}}(t_{f}))$, then\\
(2) $V^{*}({\bf{x}},t)\leq V^{(i+1)}({\bf{x}},t)\leq V^{(i)}({\bf{x}},t)$.
\end{lemma}
\begin{proof}
For the second part of \autoref{lemmav}, we should consider the properties of time-dependence and two parts of function $V({\bf{x}},t)$, $\forall {\bf{x}}\in \Omega, t\in [t_{0},t_{f}]$. We start the state ${\bf{x}}$ and go along the trajectories $\dot{{\bf{x}}}=f_{i}+g_{i}u^{(i+1)}+\sigma_{i4}\dot{B},~(i=1,2,3)$, we have
\begin{equation}
\begin{split}
&V^{(i+1)}({\bf{x}},t)-V^{(i)}({\bf{x}},t)=\int_{t}^{t_{f}}d(V^{(i)}({\bf{x}},s)-V^{(i+1)}({\bf{x}},s))\\
=&\int_{t}^{t_{f}}(V_{{\bf{x}}}^{(i)}-V_{{\bf{x}}}^{(i+1)})^{T}dx+(V_{s}^{(i)}-V_{s}^{(i+1)})ds\\
=&\int_{t}^{t_{f}}[(V_{1{\bf{x}}}^{(i)}-V_{1{\bf{x}}}^{(i+1)})^{T}(f_{1}-S(x,t)u_{1}^{(i+1)}+\sigma_{14}\dot{B})
+(V_{2{\bf{x}}}^{(i)}-V_{2{\bf{x}}}^{(i+1)})^{T}(f_{2}-\frac{m(x,\Lambda_{t})I(x,t)u_{2}^{(i+1)}}{1+\eta(x,\Lambda_{t}) I(x,t)}+\sigma_{24}\dot{B})\\
&+(V_{3{\bf{x}}}^{(i)}-V_{3{\bf{x}}}^{(i+1)})^{T}(f_{3}+S(x,t)u_{1}^{(i+1)}
+\frac{m(x,\Lambda_{t})I(x,t)u_{2}^{(i+1)}}{1+\eta(x,\Lambda_{t}) I(x,t)})+\sigma_{34}\dot{B}]+(V_{s}^{(i)}-V_{s}^{(i+1)})ds,
\end{split}
\end{equation}
where
\begin{equation}
\begin{split}
f_{1}=&(1-p(x,\Lambda_{t}))b(x,\Lambda_{t})+\alpha(x,\Lambda_{t}) I(x,t)-(\mu(x,\Lambda_{t})-\beta(x,\Lambda_{t}) S(x,t) I(x,t),\\
f_{2}=&\beta(x,\Lambda_{t}) S(x,t) I(x,t)  +(1-e)\beta(x,\Lambda_{t}) V(x,t)I(x,t)-(\mu(x,\Lambda_{t})+\alpha(x,\Lambda_{t}))I(x,t),\\
f_{3}=& p(x,\Lambda_{t}) b(x,\Lambda_{t})-\mu(x,\Lambda_{t})V(x,t) -(1-e)\beta(x,\Lambda_{t}) V(x,t) I(x,t).
\end{split}
\end{equation}
Based on LEs \eqref{lkxcc}, we have
\begin{equation}
\begin{split}
&V^{(i+1)}({\bf{x}},t)-V^{(i)}({\bf{x}},t)\\
=&\int_{t}^{t_{f}}((u_{1}^{(i+1)T}\tau_{1}u_{1}^{(i+1)}-u_{1}^{(i)T}\tau_{1}u_{1}^{(i)})
+(u_{2}^{(i+1)T}\tau_{2}u_{2}^{(i+1)}-u_{2}^{(i)T}\tau_{2}u_{2}^{(i)})\\
&+(V_{1{\bf{x}}}^{(i)T}-V_{3{\bf{x}}}^{(i)T})S(x,t)(u_{1}^{(i+1)}-u_{1}^{(i)})+(V_{2{\bf{x}}}^{(i)T}
-V_{3{\bf{x}}}^{(i)T})\frac{m(x,\Lambda_{t})I(x,t)}{1+\eta(x,\Lambda_{t}) I(x,t)}(u_{2}^{(i+1)}-u_{2}^{(i)}))ds,
\end{split}
\end{equation}
according to update control policy
\begin{equation}
u_{1}^{(i+1)}=-(2\tau_{1})^{-1}S^{T}(x,t)(V_{1{\bf{x}}}^{(i)}-V_{3{\bf{x}}}^{(i)}),~\mbox{and}~u_{2}^{(i+1)}
=-[2\tau_{2}(1+\eta(x,\Lambda_{t}) I(x,t))]^{-1}I^{T}(x,t)(V_{2{\bf{x}}}^{(i)}-V_{3{\bf{x}}}^{(i)}),
\end{equation}
then,
\begin{equation}
(V_{1{\bf{x}}}^{(i)T}-V_{3{\bf{x}}}^{(i)T})S(x,t)=-2\tau_{1}u_{1}^{(i+1)},~\mbox{and}~(V_{2{\bf{x}}}^{(i)T}
-V_{3{\bf{x}}}^{(i)T})[(1+\eta(x,\Lambda_{t}) I(x,t))]^{-1}m(x,\Lambda_{t})I(x,t)=-2\tau_{2}u_{2}^{(i+1)},
\end{equation}
and
\begin{equation}
\begin{split}
&V^{(i+1)}({\bf{x}},t)-V^{(i)}({\bf{x}},t)\\
=&\int_{t}^{t_{f}}[(u^{(i+1)T}\tau_{1}u^{(i+1)}-u^{(i)T}\tau_{1}u^{(i)})
+(u^{(i+1)T}\tau_{2}u^{(i+1)}-u^{(i)T}\tau_{2}u^{(i)})\\
&-2\tau_{1}u_{1}^{(i+1)}(u_{1}^{(i+1)}-u_{1}^{(i)})-2\tau_{2}u_{2}^{(i+1)}(u_{2}^{(i+1)}
-u_{2}^{(i)})]ds\\
=&-\int_{t}^{t_{f}}[(u_{1}^{(i+1)}-u_{1}^{(i)})^{T}\tau_{1}(u_{1}^{(i+1)}-u_{1}^{(i)})
+(u_{2}^{(i+1)}-u_{2}^{(i)})^{T}\tau_{2}(u_{2}^{(i+1)}-u_{2}^{(i)})]ds\leq 0.
\end{split}
\end{equation}
Therefore, $V^{(i+1)}({\bf{x}},t)\leq V^{(i)}({\bf{x}},t)$. By the definition of $V^{*}(x,t)$, we have
\begin{equation}
V^{*}({\bf{x}},t)\leq V^{(i+1)}({\bf{x}},t)\leq V^{(i)}({\bf{x}},t),
\end{equation}
by the monotone bounded convergence theorem, we can conclude that $V^{(i)}({\bf{x}},t)$ converges to $V^{(\infty)}({\bf{x}},t)$. In addition, $V^{(i)}({\bf{x}},t)$ is bounded below by $V^{*}({\bf{x}},t)$ and $V^{*}({\bf{x}},t)$ is the unique optimal value function. Hence, we have $V^{(\infty)}({\bf{x}},t)=V^{*}({\bf{x}},t)$. That is to say, $V^{(i)}({\bf{x}},t)$ converges uniformly to $V^{*}({\bf{x}},t)$, which is the solution to the HJB equation.
\end{proof}

%\begin{remark}
%\textcolor{red}{From Lemma 1, it is easy to obtain the approximation solution to the time-varying HJB equation related to the finite-horizon optimal control of time-delay system. Instead of solving the HJB equation (10) directly, the solution to the HJB equation can be solved based on the successive solution to the LE (6). However, the successive approximation approach mentioned above relies on the completely known system dynamics.

In practice, the accurate system dynamics are usually unavailable due to the complicated and large scale manufacturing techniques. In order to overcome the difficulty, we need to design an off-model learning algorithm.
%\end{remark}

{\bf Algorithm. Online off-policy integral reinforcement learning (IRL) algorithm}\\
{\bf{Step 1.}} Select an initial admissible control policy $u^{(0)}$ and time step $\delta$. Let $i_{\max}$ be prefixed according to $t_{f}$ and $\delta$. Let $i=0$.\\
{\bf{Step 2.}} Solving the following time-varying integral LE for $V^{(i)}$ and $u^{(i+1)}$.
\begin{equation}\label{pss}
\begin{split}
V^{(i)}({\bf{x}}(t-\delta),t-\delta)&=2\int_{t-\delta}^{t}[u^{(i+1)\delta}\tau_{1}(u-u^{(i)})
+u^{(i+1)\delta}\tau_{1}(u-u^{(i)})]ds\\
&+\int_{t-\delta}^{t}(A_{1}S(x,t)+A_{2}I(x,t)+u_{1}^{(i)}\tau_{1}u_{1}^{(i)}
+u_{2}^{(i)}\tau_{2}u_{2}^{(i)})ds+V^{(i)}({\bf{x}},t)
\end{split}
\end{equation}
with $V^{(i)}({\bf{x}}(t_{f}),t_{f})=\phi({\bf{x}}(t_{f}))$.\\
{\bf{Step 3.}}
If $i<i_{\max}$, let $i=i+1$ and go to {\bf{Step 2}}; else stop and output the approximate optimal value $V^{(i)}$ as the optimal value function $V^{*}$.

Here we derive the online off-policy IRL algorithm for optimal control of SIV epidemic system. Firstly, we rewrite the system \eqref{1`q1} as
\begin{equation}\label{5}
\small
\begin{dcases}
\begin{split}
{d{S(x,t)}}=&\bigg(D_{1}(x)\Delta S(x,t)+(1-p)b+\alpha  I(x,t)-\mu S(x,t)-\beta S(x,t) I(x,t)-(u_{1}(x,t) \\
&-u_{1}^{(i)}(x,t)) S(x,t)-u_{1}^{(i)}(x,t) S(x,t)\bigg)dt-\sigma S(x,t) I(t)dB_{1}(x,t),\\
%&\equiv f_{1}(x(t),u)dt-\sigma_{14}(x(t))dB(t), \\
{d{I}(x,t)}=& \bigg(D_{2}(x)\Delta I(x,t)+\beta S(x,t) I(x,t)  +(1-e)\beta I(x,t) V(x,t)-(\mu+\alpha )I(x,t)\\
&-\frac{m(u_{2}(x,t) -u_{2}^{(i)}(x,t))I(x,t)}{1+\eta I(x,t)}-\frac{mu_{2}^{(i)}(x,t)I(x,t)}{1+\eta I(x,t)}\bigg)dt\\
&+\sigma S(t) I(t)dB_{2}(t)+\sigma I(t) V(t)dB_{4}(x,t),\\
\end{split}
\end{dcases}
\end{equation}

\begin{equation*}
\small
\begin{dcases}
\begin{split}
%&\equiv f_{2}(x(t),u)dt-\sigma_{24}(x(t))dB(t), \\
{d{V}(x,t)} =& \bigg( D_{3}(x)\Delta V(x,t)+p b+(u_{1}(x,t) -u_{1}^{(i)}(x,t)) S(x,t)-u_{1}^{(i)}(x,t) S(x,t)-\mu V(x,t)\\
&-(1-e)\beta I(x,t) V(x,t)+\frac{m(u_{2}(x,t) -u_{2}^{(i)}(x,t))I(x,t)}{1+\eta I(x,t)}\\
&-\frac{mu_{2}^{(i)}I(x,t)}{1+\eta I(x,t)}\bigg)dt-\sigma I(x,t) V(x,t)dB_{3}(x,t),\\
%&\equiv f_{3}(x(t),u)dt-\sigma_{34}(x(t))dB(t),\\
S(0,x)=&S_{0}(x),~~ I(0,(x,t))=I_{0}(x),~~ V(0,x)=V_{0}(x).
\end{split}
\end{dcases}
\end{equation*}

Let us consider $V^{(i)}$, which is the solution to the Lyapunov equation $LE(V^{(i)},u^{(i)})=0$. Then, we have
\begin{equation*}
\begin{split}
0=V_{t}^{(i)}&+p_{1}^{i}[D_{1}(x)\Delta S(x,t)+(1-p)b+\alpha  I(x,t)-\mu S(x,t)-\beta S(x,t) I(x,t)]\\
&+p_{2}^{(i)}[D_{2}(x)\Delta I(x,t)+\beta S(x,t) I(x,t)  +(1-e)\beta I(x,t) V(x,t)-(\mu+\alpha )I(x,t)]\\
&+p_{3}^{(i)}[D_{3}(x)\Delta V(x,t)+ p b-\mu V(x,t)-(1-e)\beta I(x,t) V(x,t)]\\
&+\frac{1}{2}[P_{1}^{(i)}\sigma^{2}S^{2}(x,t)I^{2}(x,t)+P_{2}^{(i)}(\sigma^{2}S^{2}(x,t)I^{2}(x,t)
+\sigma^{2}I^{2}(x,t)V^{2}(x,t))
+P_{3}^{(i)}\sigma^{2}I^{2}(x,t)V^{2}(x,t)]\\
&+A_{1}S(x,t)+A_{2}I(x,t)+p_{3}^{(i)}u_{1}^{i}S(x,t)-p_{1}^{i}u_{1}^{i}S(x,t)+\tau_{1}u_{1}^{2}+p_{3}\frac{mu_{2}I}{1+\eta I(x,t)}-p_{2}\frac{mu_{2}I(x,t)}{1+\eta I(x,t)}+\tau_{2}u_{2}^{2}\\
&+(p_{3}^{(i)}-p_{1}^{(i)})(u_{1}-u_{1}^{(i)})S(x,t)+\tau_{1}(u_{1}^{2}-u_{1}^{2(i)})+(p_{3}^{(i)}
-p_{2}^{(i)})\frac{m(u_{2}-u_{2}^{(i)})I(x,t)}{1+\eta I(x,t)}+\tau_{2}(u_{2}^{2}-u_{2}^{2(i)}),
\end{split}
\end{equation*}

Since $V^{(i)}$ is the solution to the Lyapunov equation \eqref{lkxcc}, we have
\begin{equation}\label{pagc}
\dot{V}^{(i)}({\bf{x}}(t),t)=-A_{1}S(x,t)-A_{2}I(x,t)-u_{1}^{(i)T}\tau_{1}u_{1}^{(i)}
-u_{2}^{(i)T}\tau_{2}u_{2}^{(i)}+V_{x}^{(i)T}g(u-u^{(i)}).
\end{equation}
Integrating the both sides of equation \eqref{pagc}, we can get
\begin{equation}
\begin{split}
&V^{(i)}({\bf{x}}(t),t)-V^{(i)}({\bf{x}}(t-\delta),t-\delta)\\&=\int_{t-\delta}^{t}(V_{{\bf{x}}}^{(i)T}g(u-u^{(i)})
-A_{1}S(x,t)-A_{2}I(x,t)-u_{1}^{(i)T}\tau_{1}u_{1}^{(i)}-u_{2}^{(i)T}\tau_{2}u_{2}^{(i)})ds.
\end{split}
\end{equation}

\begin{lemma}
Solving the $(V^{(i)},u^{(i+1)})$ of LEs \eqref{pss} is equivalent to finding the solution to the following equations
$LE(V^{(i)},u^{(i)})=0~ \mbox{and}~ u_{1}^{(i+1)}=-(2c_{2})^{-1}S^{T}(x,t)(V_{1{\bf{x}}}^{(i)}-V_{3{\bf{x}}}^{(i)}),~\mbox{and}~u_{2}^{(i+1)}=-[2c_{3}(1+\eta I)]^{-1}I^{T}(x,t)(V_{2{\bf{x}}}^{(i)}-V_{3{\bf{x}}}^{(i)}),~\mbox{with}~V^{(i)}({\bf{x}}(t_{f}),t_{f})=\phi({\bf{x}}(t_{f})).$
\end{lemma}
\begin{proof}
We prove the Eq. \eqref{pss} has unique solution. This can be completed by contradiction. First, we suppose that there exists another solution $(V_{1},u'_{1},u'_{2})$
\begin{equation}\label{paa1}
\begin{split}
V_{1}({\bf{x}}(t-\delta),t-\delta)=&2\int_{t-\delta}^{t}[u_{1}^{'T}\tau_{1}(u_{1}
-u_{1}^{(i)})+u_{2}^{'T}\tau_{2}(u_{2}-u_{2}^{(i)})]ds\\
&+\int_{t-\delta}^{t}(A_{1}S(x,t)+A_{2}I(x,t)+u_{1}^{(i)T}\tau_{1}u_{1}^{(i)}
+u_{2}^{(i)T}\tau_{2}u_{2}^{(i)})ds+V_{1}({\bf{x}}(t),t).
\end{split}
\end{equation}
with $V_{1}({\bf{x}}(t_{f}),t_{f})=\phi({\bf{x}}(t_{f}))$. Let us subtract \eqref{paa1} from \eqref{pss}, yields
\begin{equation}\label{plsdf}
\begin{split}
&V^{(i)}({\bf{x}}(t-\delta),t-\delta)-V_{1}({\bf{x}}(t-\delta),t-\delta)\\
&=2\int_{t-\delta}^{t}(u_{1}^{(i+1)}-u'_{1})^{T}\tau_{1}(u_{1}-u_{1}^{(i)})ds+V^{(i)}({\bf{x}}(t),t)
-V_{1}({\bf{x}}(t),t).
\end{split}
\end{equation}
%It is known that the Eq. \eqref{plsdf} holds $\forall u$ and $\forall t, T$.
Thus, we select $u=u^{(i)}$, $V^{(i)}({\bf{x}}(t_{f}),t_{f})=\phi({\bf{x}}(t_{f}))$ and $V_{1}({\bf{x}}(t_{f}),t_{f})=\phi({\bf{x}}(t_{f}))$. Then we have $V^{(i)}({\bf{x}}(t),t)=V_{1}({\bf{x}}(t),t), \forall t \in [t_{0},t_{f}]$. It is enough to show that
\begin{equation}\label{wksg1}
\int_{t-\delta}^{t}(u_{1}^{(i+1)}-u'_{1})^{T}\tau_{1}(u_{1}-u_{1}^{(i)})ds=0,
\end{equation}
Dividing $T$ and taking limit for the both sides of Eq. \eqref{wksg1}, it follows that
\begin{equation}
\lim_{\delta\to 0}\frac{\int_{t-\delta}^{t}(u_{1}^{(i+1)}-u'_{1})^{T}\tau_{1}(u_{1}-u_{1}^{(i)})ds}{\delta}=0,
\end{equation}
that is $(u_{1}^{(i+1)}-u'_{1})^{\top}\tau_{1}(u_{1}-u_{1}^{(i)})=0$ for all $u$. Let $(u_{1}-u_{1}^{(i)})=(u_{1}^{(i+1)}-u'_{1})$. Then, we can get $u_{1}^{(i+1)}=u'_{1}$.\\ This completes the proof.
\end{proof}

\section{Numerical examples}\label{Section6}
Applying Milstein's method, we have the corresponding  equation of state equation \eqref{1`q1} and adjoint equation \eqref{adj} as follows:
\begin{equation*}\label{op}
\begin{cases}
\begin{split}
S_{i+1} =& S_i + \bigg((1-p(\Lambda_{t}))b(\Lambda_{t})+\alpha(\Lambda_{t})  I_i-\mu(\Lambda_{t}) S_{i}-\beta(\Lambda_{t}) S_i I_i-u_{1} S_i+D_{1}\Delta S_{i}      \bigg) \Delta t\\
&-\sigma(\Lambda_{t}) S_i I_i\sqrt{\Delta t} \zeta_i  -\frac{1}{2}  \sigma^{2}(\Lambda_{t}) S_i I_i^{2}(\zeta^{2}_i-1) \Delta t, \\
I_{i+1} =& I_i + \bigg[\beta(\Lambda_{t}) S_i I_i  +(1-e)\beta(\Lambda_{t}) V_i I_i-(\mu(\Lambda_{t})+\alpha(\Lambda_{t}) )I_i-\frac{m(\Lambda_{t}) u_{2} I_i}{1+\eta(\Lambda_{t}) I_i}+D_{2}\Delta I_{i}\bigg]\Delta t
\\&+\sigma(\Lambda_{t}) S_{i} I_{i} \sqrt{\Delta t} \zeta_i+\frac{1}{2}\sigma^{2}(\Lambda_{t}) S_i^{2} I_i(\zeta^{2}-1)\Delta t+(1-e)\sigma(\Lambda_{t}) V_i I_i \sqrt{\Delta t} \zeta_i +\frac{1}{2}\sigma^{2}(\Lambda_{t})V_i^{2} I_i (\zeta^{2}-1)\Delta t, \\
V_{i+1} =& V_i + \bigg[p(\Lambda_{t}) b(\Lambda_{t})-\mu(\Lambda_{t}) V_{i}-(1-e)\beta(\Lambda_{t}) V_i I_i+u_{1} S_i+\frac{m(\Lambda_{t}) u_{2} I_i}{1+\eta(\Lambda_{t}) I_i}+D_{3}\Delta V_{i} \bigg]\Delta t
\\& -(1-e)\sigma(\Lambda_{t}) V_i I_i \sqrt{\Delta t} \zeta_i-\frac{1}{2}(1-e)^{2}\sigma^{2}(\Lambda_{t})V_{i} I_{i}^{2}(\zeta_i^2-1)\Delta t, \\
\end{split}
\end{cases}
\end{equation*}
\begin{equation*}\label{yu}
\begin{cases}
\begin{split}
p_{1_{i}} =& p_{1_{i+1}}  - \bigg[\bigg((\mu(\Lambda_{t})+u_1)S_{i+1}+\beta(\Lambda_{t}) I_{i+1}   \bigg) p_{1_{i+1}}+D_{1}\Delta p_{1_{i+1}}+\beta(\Lambda_{t}) I  _{i+1} p_{2_{i+1}}
\\&+u_1  p_{3_{i+1}}-\sigma(\Lambda_{t}) I  _{i+1}q_{1_{i+1}}+\sigma(\Lambda_{t}) I _{i+1} q_{2_{i+1}}\bigg]\Delta t  - q_{1_{i+1}}\sqrt{\Delta t} \zeta_{i+1} - \frac{q^{2}_{1_{i+1}}}{2}(\zeta_{i+1}^2-1)\Delta t,\\
p_{2_{i}} =& p_{2_{i+1}} -\bigg[(\alpha(\Lambda_{t})-\beta(\Lambda_{t}) S_{i+1})p_{1_{i+1}}+\bigg(\beta(\Lambda_{t}) S_{i+1}+  (1-e)\beta (\Lambda_{t})V_{i+1}-(\mu(\Lambda_{t})
\\&+\alpha(\Lambda_{t}))-\frac{m(\Lambda_{t}) u_{2}(x,t)}{(1+\eta(\Lambda_{t}) I)^{2}}\bigg)p_{2_{i+1}}+D_{2}\Delta p_{2_{i+1}}-\bigg((1-e)\beta(\Lambda_{t}) V_{i+1}-\frac{m(\Lambda_{t}) u_{2}}{(1+\eta (\Lambda_{t})I)^{2}}\bigg) p_{3_{i+1}}
\\&-\sigma(\Lambda_{t}) S_{i+1} q_{1 _{i+1}}+(\sigma(\Lambda_{t}) S_{i+1}+(1-e)\sigma(\Lambda_{t}) V_{i+1})q_{2_{i+1}}-(1-e)\sigma(\Lambda_{t}) V  _{i+1}  q_{3_{i+1}}\bigg] \Delta t
\\&  - q_{2_{i+1}}\sqrt{\Delta t} \zeta_{i+1} - \frac{q^{2}_{2_{i+1}}}{2}(\zeta_{i+1}^2-1)\Delta t,\\
p_{3_{i}} =& p_{3_{i+1}} + \bigg[(1-e)\beta(\Lambda_{t}) I  _{i+1} p_{2_{i+1}}  -\bigg(\mu(\Lambda_{t}) +(1-e)\beta(\Lambda_{t})  I _{i+1}\bigg)p_{3_{i+1}}+D_{3}\Delta p_{3_{i+1}}
\\&+(1-e)\sigma(\Lambda_{t})  I  _{i+1}q_{2_{i+1}}-(1-e)\sigma (\Lambda_{t}) I  _{i+1}q_{3_{i+1}}\bigg]\Delta t,\\
\end{split}
\end{cases}
\end{equation*}
where $\zeta_i^2 ~(i=1,2,...)$  are not interdependent Gaussian random variables $N(0,1)$.
At below, we will present numerical simulation of the SIV model, let us assume that the Markov chain $\xi(t)$ is on the state space $\mathbb{S}=\{1,2\}$ with the transition probability $\alpha=\left(
\begin{array}{ccc}
-5.5  & 5.5   \\
8  & -8  \\
%\frac{7}{8}  & \frac{1}{8} & 0  \\
\end{array}
\right)
$ and the following setting:

When $\xi(t)=1$,\\
$p(1)=0.5year^{-1},~b(1)=4.0year^{-1},~\beta(1)=0.02year^{-1},~\mu(1)=0.04year^{-1},~\eta(1)=1.03year^{-1},~m(1)=0.01year^{-1},~\alpha(1)=0.001year^{-1},~\alpha(1)=0.8year^{-1},~\sigma(1)=0.035year^{-1};$

When $\xi(t)=2$,\\
$p(2)=0.6year^{-1},~b(2)=5.0year^{-1},~\beta(2)=0.04year^{-1},~\mu(2)=0.05year^{-1},~\eta(2)=1.05year^{-1},~m(2)=0.02year^{-1},~\alpha(2)=0.002year^{-1},~\alpha(2)=0.9year^{-1},~\sigma(2)=0.036year^{-1}.$

%When $\xi(t)=3$,\\
%$p(3)=0.7,~b(3)=5.5,~\beta(3)=0.05,~\mu(3)=0.06,~\eta(3)=1.10,~m(3)=0.03,~\alpha(3)=0.003,~\alpha(3)=1.0,~\sigma(3)=0.038.$
\subsection{The invariant measure}
\begin{figure}[H]\centering
\begin{subfigure}[]{
			\includegraphics[scale=0.23]{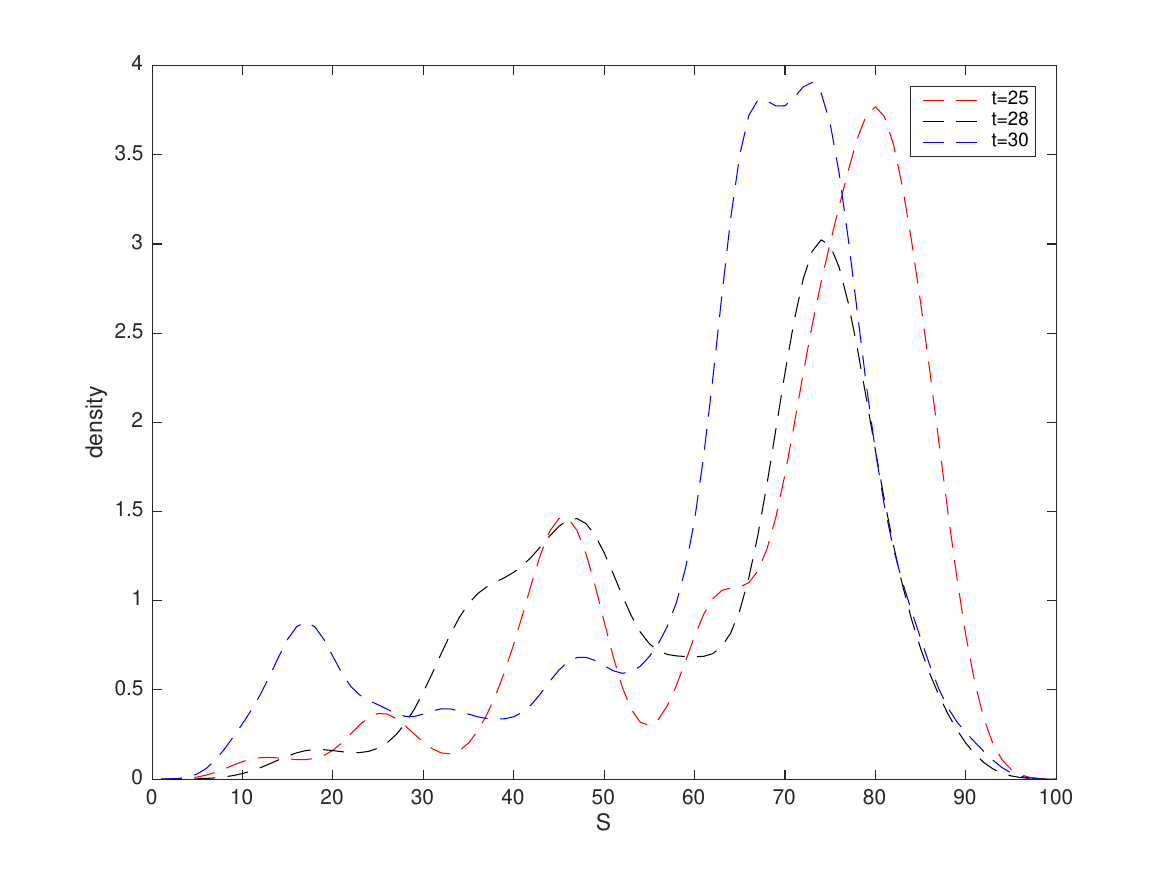}}\end{subfigure}
\begin{subfigure}[]{
			\includegraphics[scale=0.23]{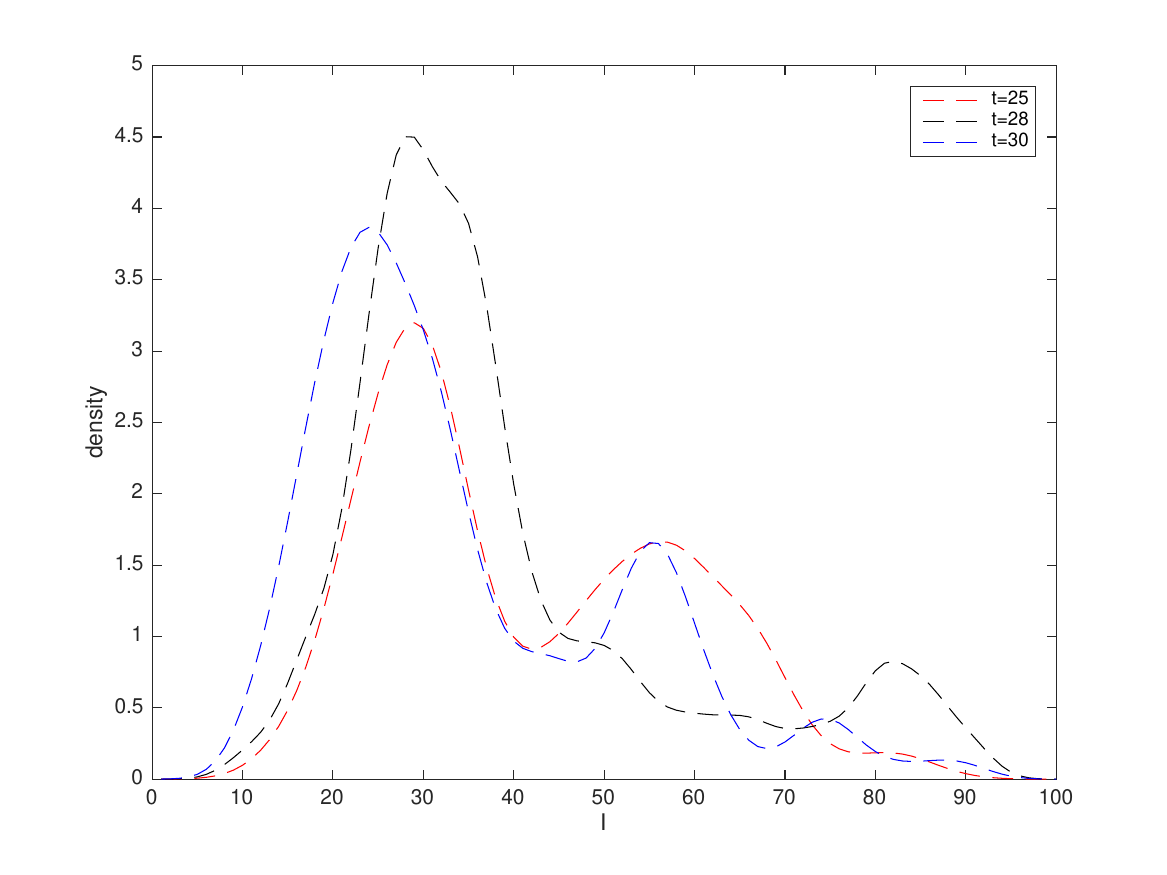}}\end{subfigure}
\begin{subfigure}[]{
			\includegraphics[scale=0.23]{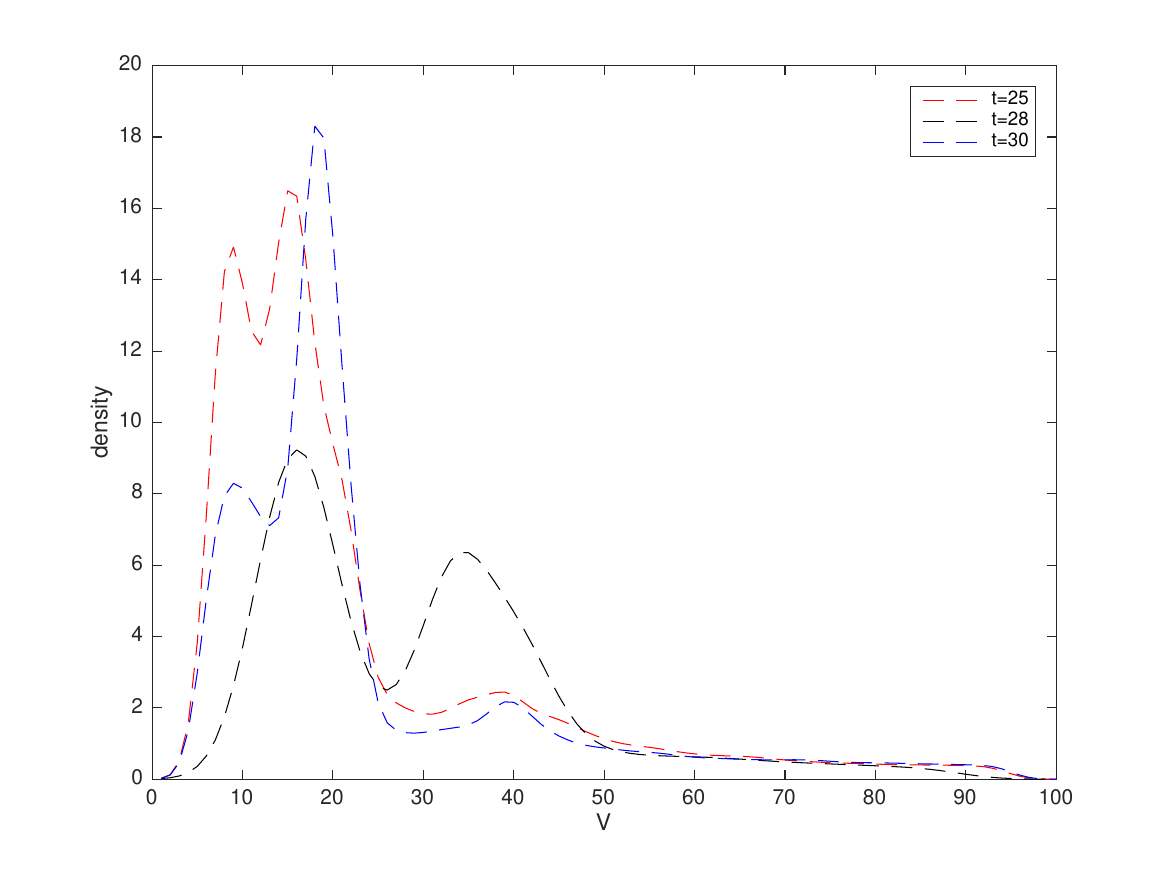}}\end{subfigure}
	\caption{Density plots (a)-(c) based on 10 000 stochastic simulations for group susceptible, infected and vaccinated at time t = 25, 28 and 30. Here we choose $\beta=0.2,~\sigma=0.05$. The simulations confirm the existence of the unique ergodic invariant measure for system \eqref{adj}.}\label{cc}
\end{figure}
In \autoref{cc}, we simulate the density kernels of solutions \eqref{1`q1} with three groups namely $(S,I,V)$. In the simulation, the density kernels are based on 10 000 sample paths. Comparing these density kernels, we can see that the density plot of each group at different time $t$ for $t = 160, 180, 200$ stay almost the same. Therefore, we can conclude that the simulations strongly indicate the existence of the unique ergodic invariant measure for the system \eqref{1`q1}. In \autoref{nn}, we simulate the density kernels of solutions \eqref{1`q1} with three groups namely $ (S,I,V)$ with and without control.

\begin{figure}[H]\centering
\begin{subfigure}[]{
			\includegraphics[scale=0.23]{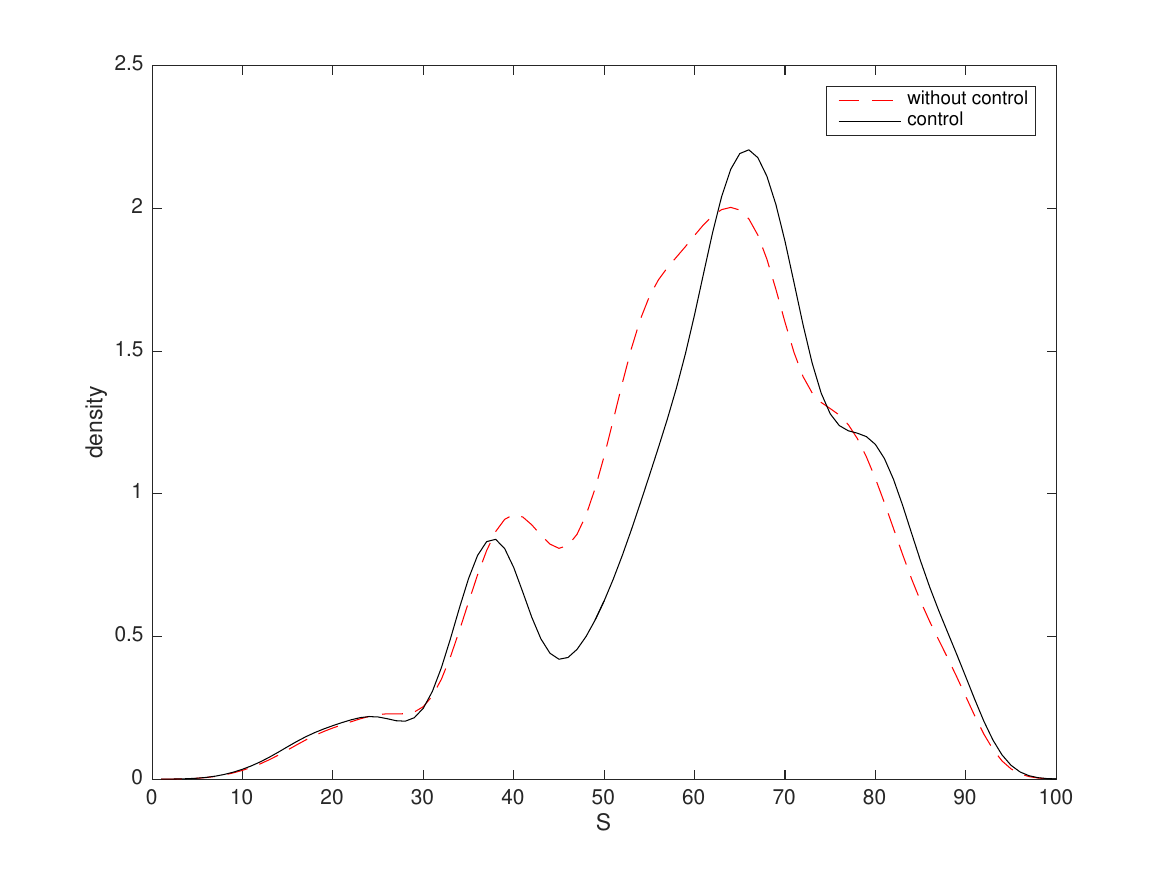}}\end{subfigure}
\begin{subfigure}[]{
			\includegraphics[scale=0.23]{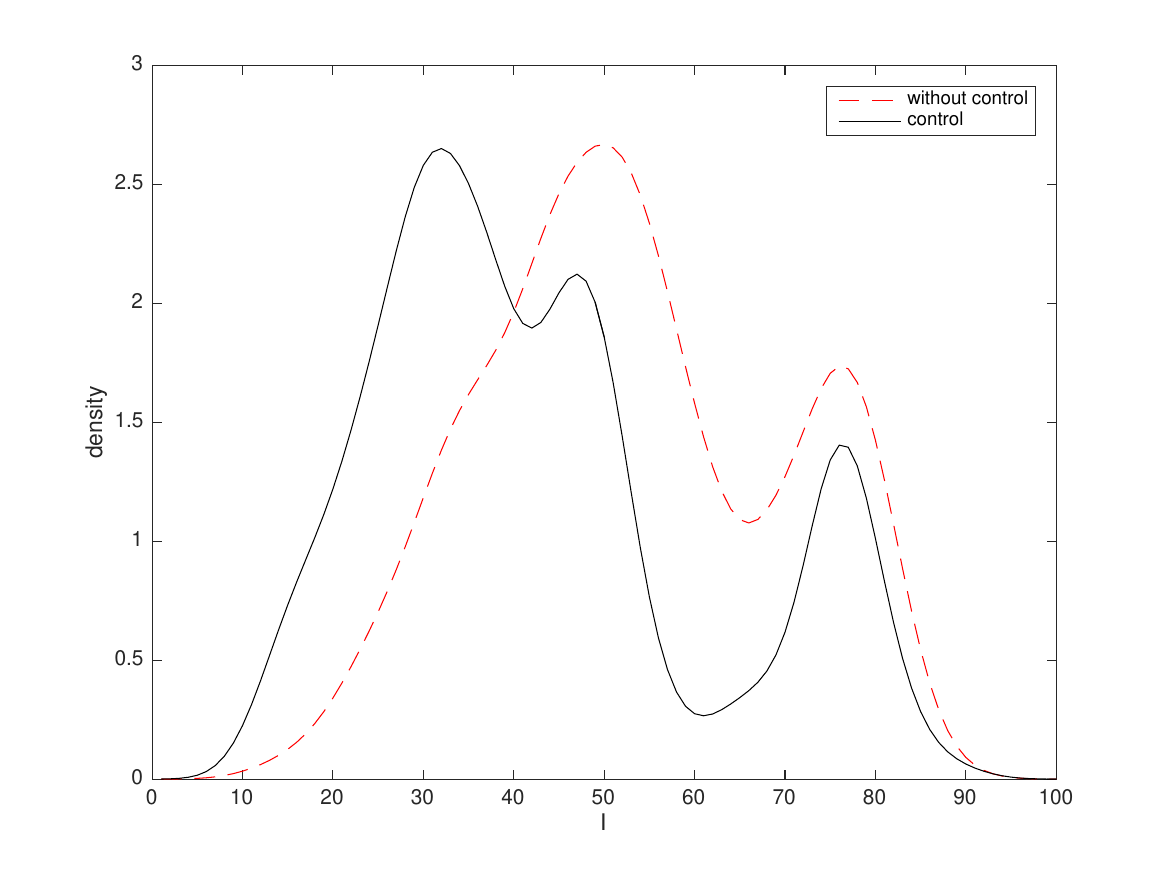}}\end{subfigure}
\begin{subfigure}[]{
			\includegraphics[scale=0.23]{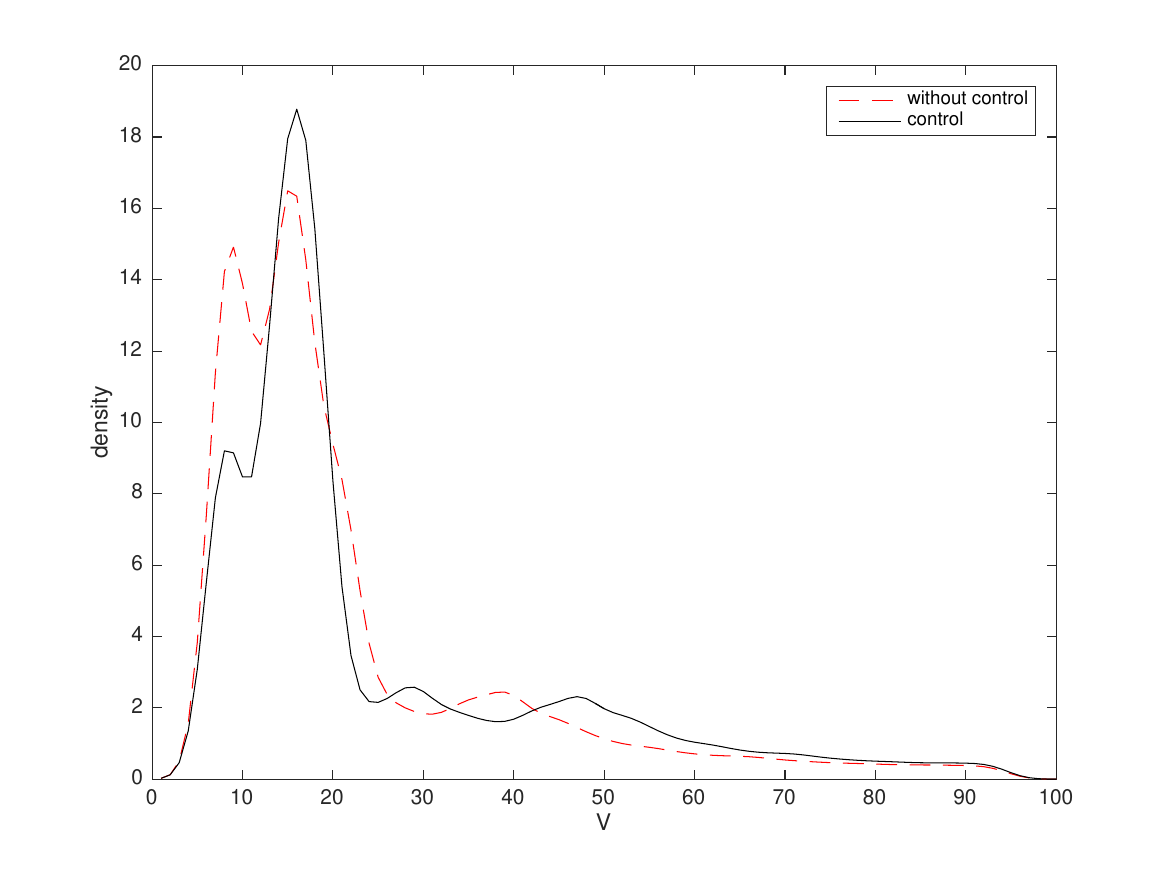}}\end{subfigure}
	\caption{Density plots (a)-(c) based on 10 000 stochastic simulations for group susceptible, infected and vaccinated with and without control. Here we choose $\beta=0.2,~\sigma=0.05$. The simulations confirm the existence of the unique ergodic invariant measure for system \eqref{adj}.}\label{fig:wsr}
\label{nn}
\end{figure}

\subsection{Finite horizon optimal control}
In \autoref{fig:wsrs}, the three solution curves represent susceptible and infected  individuals at different intervals of time with Markov switching. The application of vaccination and treatment control give rise to the number of individuals vaccinated. In \autoref{vxad}, the three solution curves represent adjoint equations corresponding to the state equations \eqref{1`q1} at different intervals of time with Markov switching. \autoref{fig:wsrt} gives the optimal control $u_{1}$ and $u_{2}$, and also give the value function of the stochastic SIV epidemic system with Markov switching.

\begin{figure}[h]\centering
\begin{subfigure}[]{
			\includegraphics[scale=0.3]{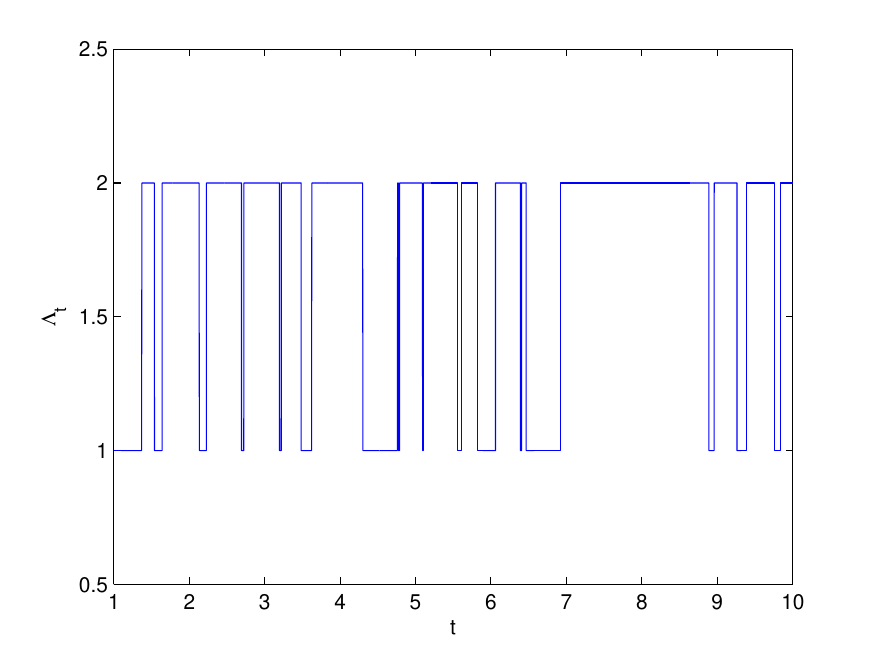}}\end{subfigure}
\begin{subfigure}[]{
			\includegraphics[scale=0.32]{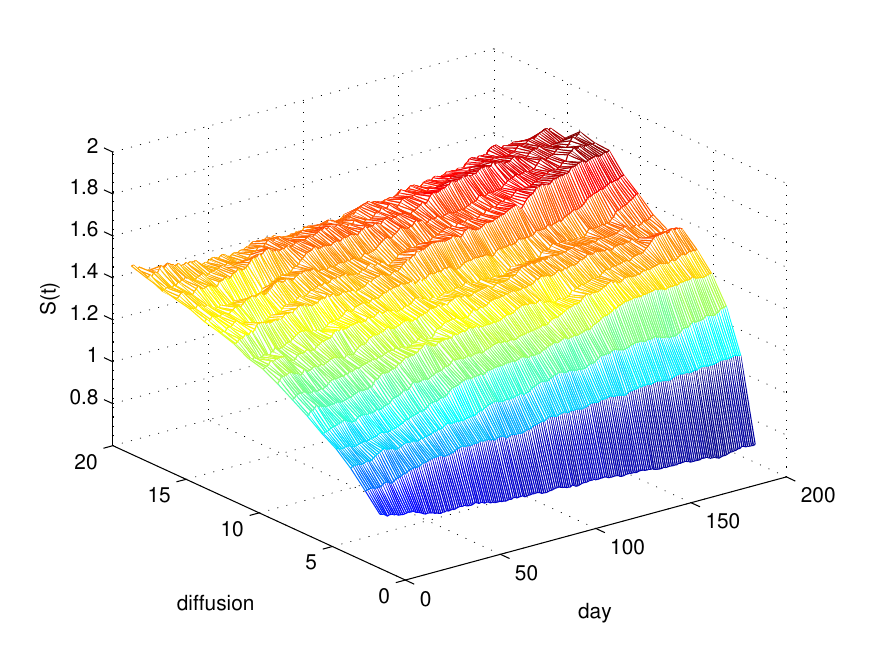}}\end{subfigure}
\begin{subfigure}[]{
			\includegraphics[scale=0.32]{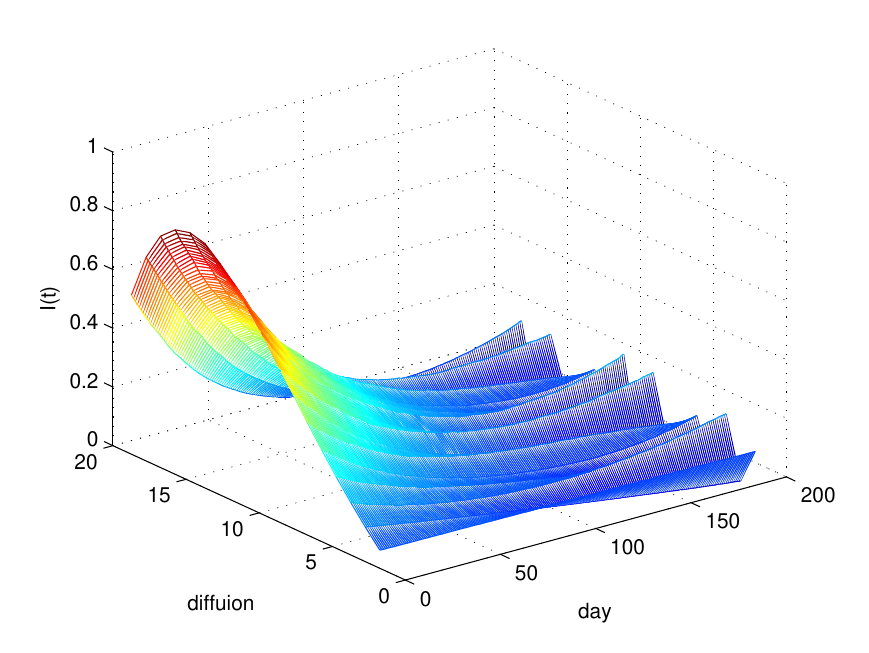}}\end{subfigure}
	\caption{The path of $S(x,t), I(x,t)$  for the stochastic SIV model \eqref{1`q1} with initial $S_0 = 0.6, I_0 = 0.1,V_0 = 1.0$.}\label{fig:wsr}
\label{fig:wsrs}
\end{figure}

\begin{figure}[h]\centering
\begin{subfigure}[]{
			\includegraphics[scale=0.23]{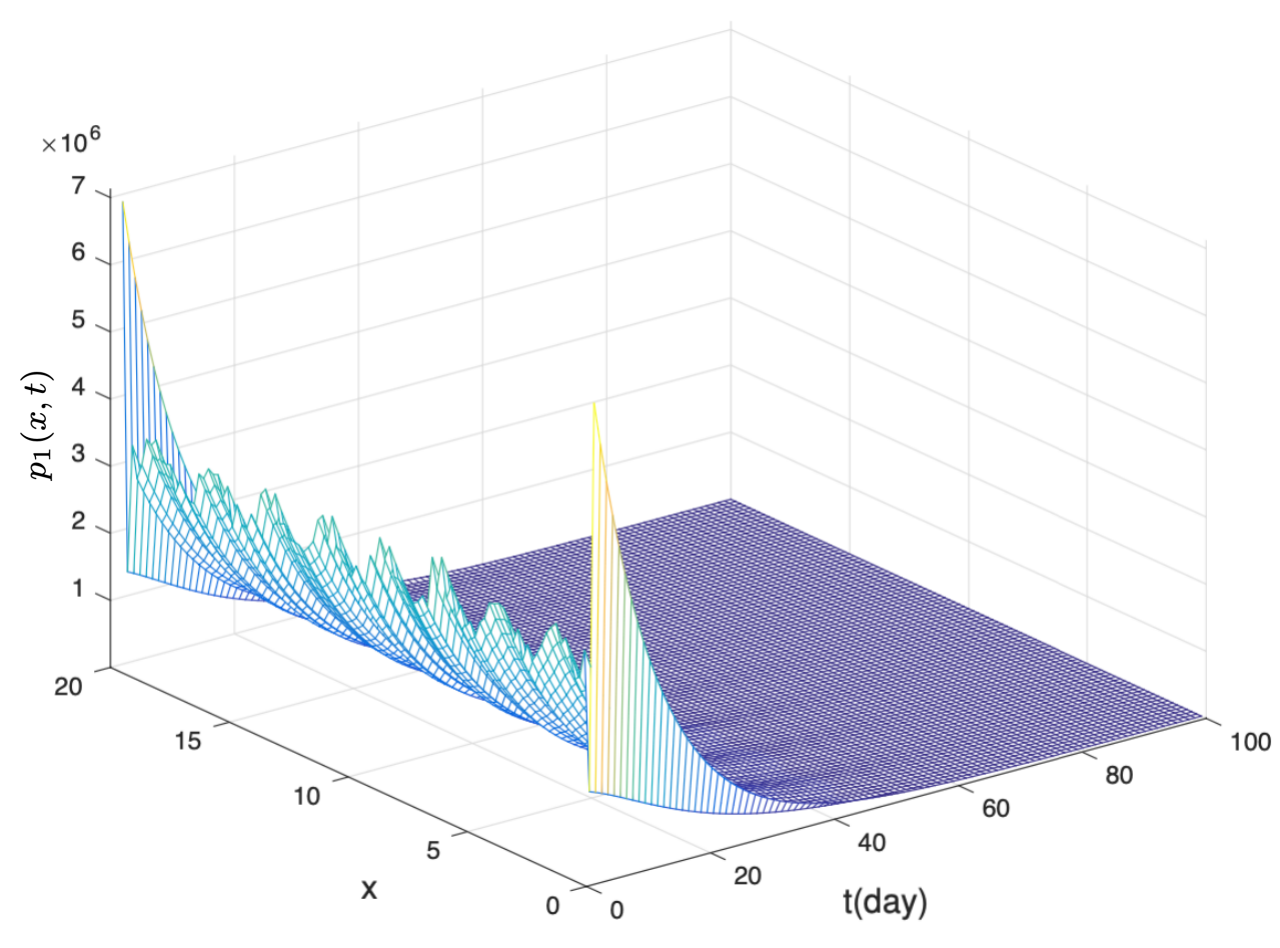}}\end{subfigure}
\begin{subfigure}[]{
			\includegraphics[scale=0.2]{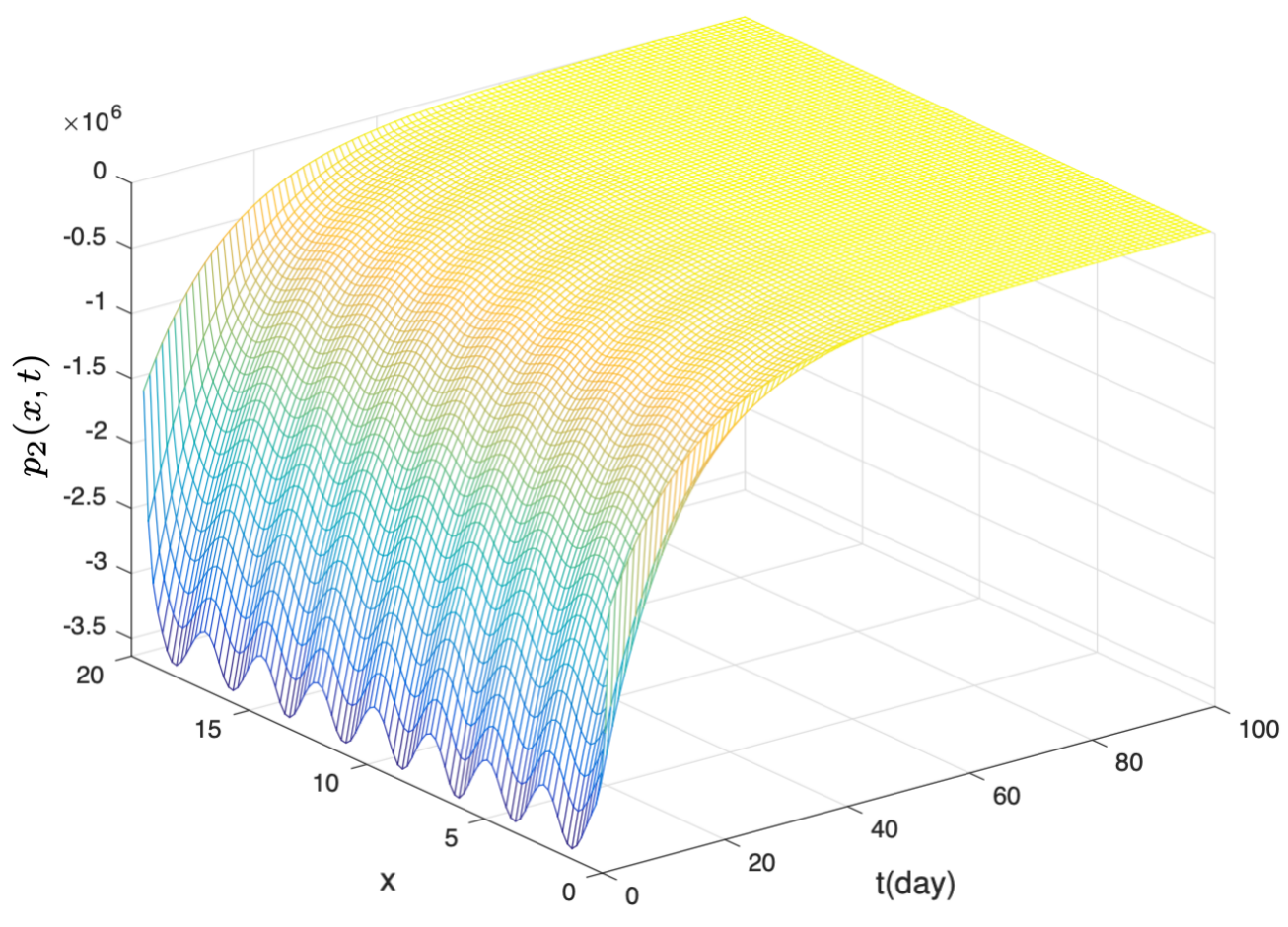}}\end{subfigure}
\begin{subfigure}[]{
			\includegraphics[scale=0.23]{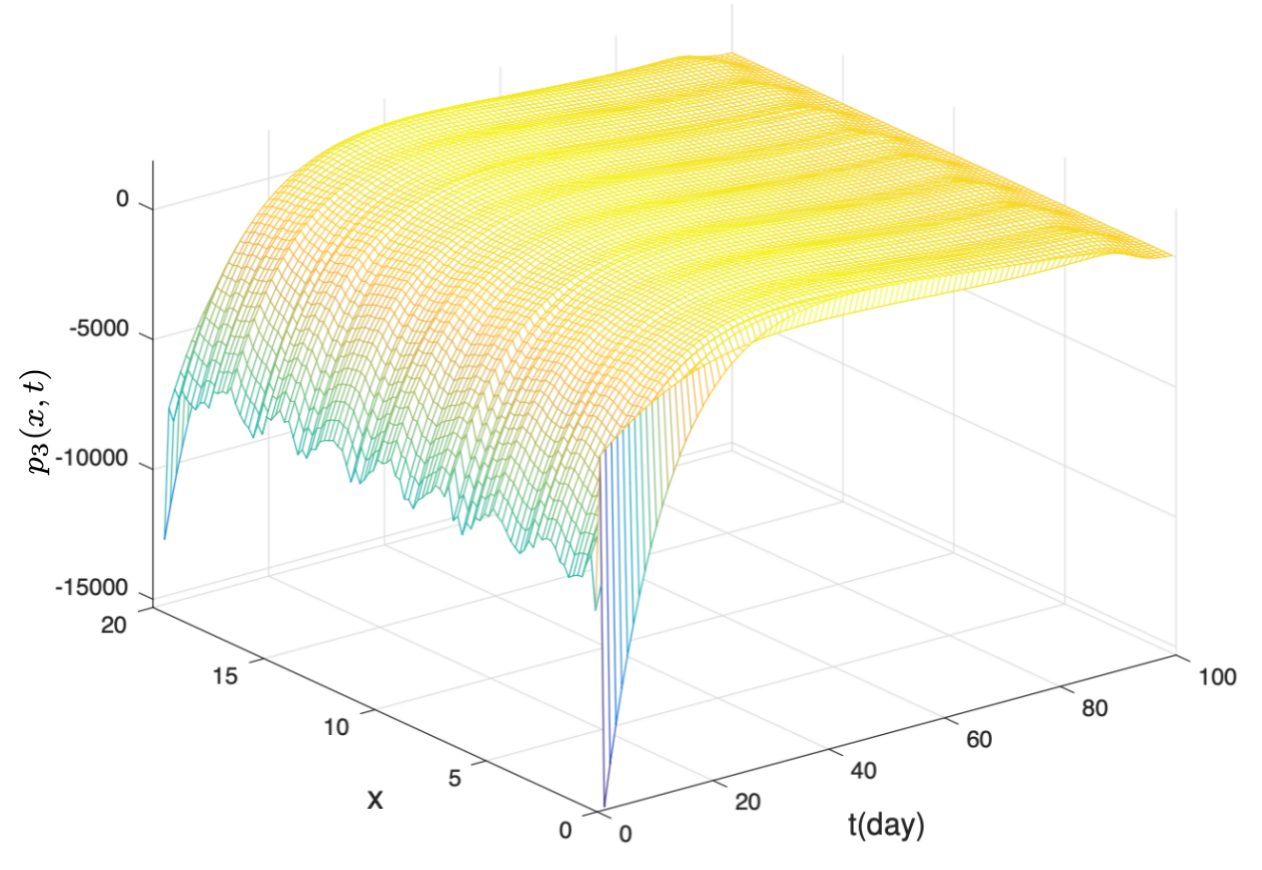}}\end{subfigure}
	\caption{The path of $p_{1}(x,t), p_{2}(x,t)$ and $p_{3}(x,t)$  of the adjoint equations \eqref{adj}.}\label{vxad}
\end{figure}

\begin{figure}[H]\centering
\begin{subfigure}[]{
			\includegraphics[scale=0.3]{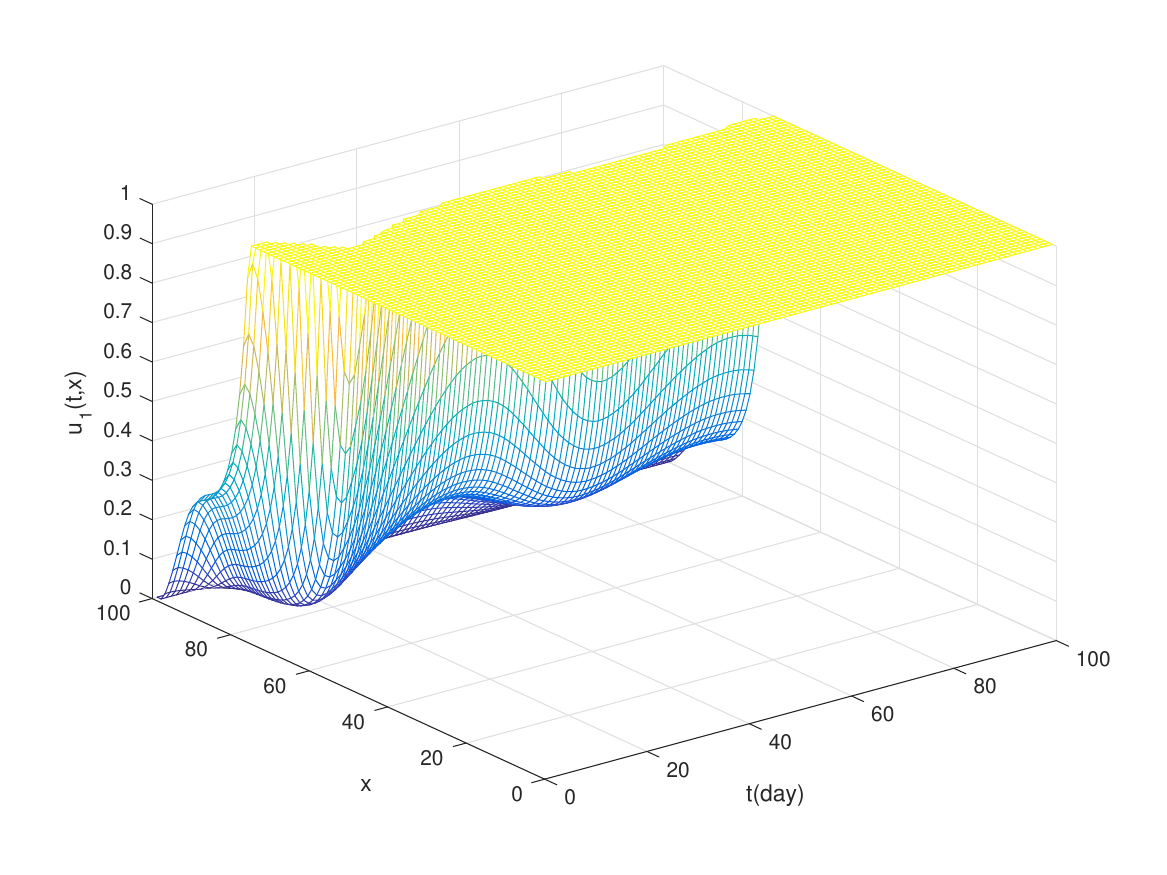}}\end{subfigure}
\begin{subfigure}[]{
			\includegraphics[scale=0.27]{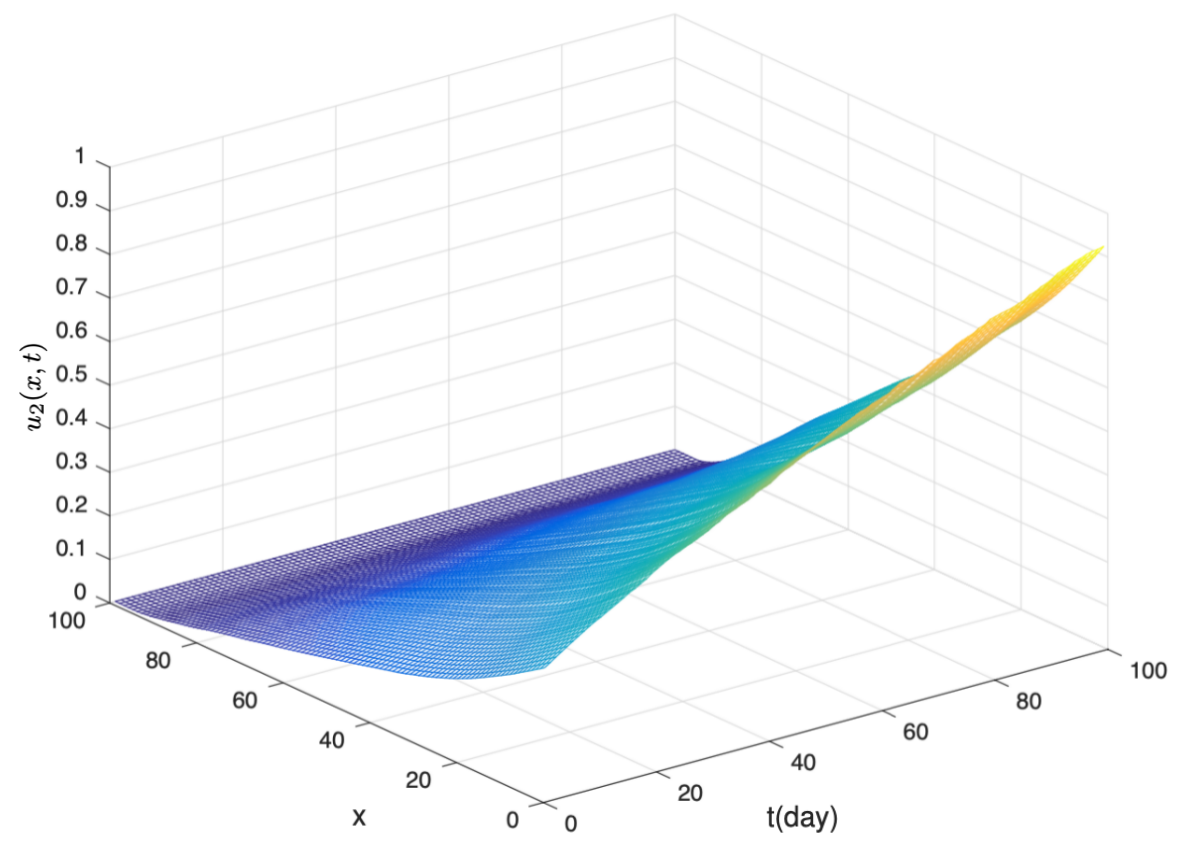}}\end{subfigure}
%\begin{subfigure}[]{
%			\includegraphics[scale=0.4]{u1111.eps}}\end{subfigure}
	\caption{The path of $u_{1}(x,t), u_{2}(x,t)$.}\label{fig:wsr}
\label{fig:wsrt}
\end{figure}

\section{Conclusion}\label{Section7}

In this paper, we establish a new spatial diffusion stochastic SIV epidemic  model with Markov chain in a stochastic environment. Based on the Perron-Frobenius theorem,  we investigate the criterion on the existence and uniqueness of invariant measure for the exact solution. And we also derive a set of necessary conditions for near-optimal controls. Based on Lyapunov equation, an online off-strategy IRL algorithm for optimal control of SIV epidemic systems is derived. The existence of an invariant measure for a stochastic SIV epidemic model with Markov switching has been studied. Finally, we present a numerical simulation of the optimal control problem, which shows that different control strategies are needed at different times to minimize the objective function. The uniqueness of invariant measure of stochastic SIV epidemic model with Markov switching is verified by numerical examples.

Nevertheless, some mathematical details need to be carefully worked out. The results, in turn, will be of interest to people working on real data. On could work with stochastic partial differential equation models driven by L\'evy noise.

\section*{Data Availability Statement}
The manuscript has no associated data.

%\section*{Declarations}

\section*{Conflict of interest}
The authors declare that they have no conflict of interest.

%\section*{Acknowledgments}
%The research is supported by the Natural Science Foundation of China (Grant numbers 12161068).

\end{document}